\documentclass[11pt]{article}
\usepackage{epigamath}

\usepackage[notext]{kpfonts}
\usepackage{baskervald}

\setpapertype{A4}

\usepackage[french]{babel}
\usepackage[dvips]{graphicx}     % Use this instead with dvips

\title{Troisi\`eme groupe de cohomologie non ramifi\'ee des torseurs universels sur les surfaces rationnelles}
\titlemark{Troisi\`eme groupe de cohomologie non ramifi\'ee des torseurs universels sur les surfaces rationnelles}
\author{\vspace{0cm} Yang Cao}
\authoraddresses{
\authordata{Yang Cao}{\firstname{Yang} \lastname{Cao}\\
\institution{Laboratoire de Math\'ematiques d'Orsay,
Universit\'e Paris-Sud, CNRS, Universit\'e Paris-Saclay, 91405 Orsay, France}\\
\email{yangcao1988@gmail.com}}
}
\authormark{Y. Cao}
\date{\vspace{-5ex}} % Empty date or tweak it according to your needs
\journal{\'Epijournal de G\'eom\'etrie Alg\'ebrique} % Epijournal name
\acceptation{Received by the Editors on May 3, 2017, and in final form
on July 18, 2018. \\ Accepted on October 16, 2018.}

\newcommand{\articlereference}{Volume 2 (2018), Article Nr. 12}

 \usepackage[all]{xy}

\allowdisplaybreaks

\numberwithin{equation}{numsection}
\providecommand{\U}[1]{\protect\rule{.1in}{.1in}}
\RequirePackage{amsmath}
\RequirePackage{amssymb}
\newcommand{\BA}{{\mathbb {A}}}

\newcommand{\BG}{{\mathbb {G}}}
\newcommand{\BH}{{\mathbb {H}}}

\newcommand{\BN}{{\mathbb {N}}}

\newcommand{\BP}{{\mathbb {P}}}
\newcommand{\BQ}{{\mathbb {Q}}}

\newcommand{\BZ}{{\mathbb {Z}}}

\newcommand{\CD}{{\mathcal {D}}}

\newcommand{\CH}{{\mathcal {H}}}

\newcommand{\CK}{{\mathcal {K}}}

\newcommand{\CO}{{\mathcal {O}}}

\newcommand{\CT}{{\mathcal {T}}}

\newcommand{\CZ}{{\mathcal {Z}}}

\newcommand{\RP}{{\mathrm {P}}}

\newcommand{\Br}{{\mathrm{Br}}}

\newcommand{\codim}{{\mathrm{codim}}}
\newcommand{\Coker}{{\mathrm{Coker}}}
\newcommand{\coker}{{\mathrm{coker}}}

\newcommand{\Div}{{\mathrm{Div}}}
\newcommand{\Pic}{{\mathrm{Pic}}}

\newcommand{\Gal}{{\mathrm{Gal}}}
\newcommand{\Hom}{{\mathrm{Hom}}}

\newcommand{\Ker}{{\mathrm{Ker}}}

\newcommand{\Res}{{\mathrm{Res}}}

\newcommand{\Spec}{{\mathrm{Spec}}}

\renewcommand{\Im}{{\mathrm{Im}}}
\newcommand{\CCH}{{\mathrm{CH}}}
\newcommand{\nr}{{\mathrm{nr}}}
\newcommand{\Sym}{{\mathrm{Sym}}}
\newcommand{\symprod}{\cdot}

\newcommand{\ra}{\rightarrow}
\newcommand{\iso}{\stackrel{\sim}{\rightarrow} }
\newcommand{\lto}{\longmapsto}

\newcommand{\sbt}{\subset}
\newcommand{\bk}{\bar{k}}

\newcommand{\et}{{\rm{\acute et}}}

\newtheorem{defi}{D\'efinition}[section]

\newtheorem{exam}[defi]{Example}

\newtheorem{rem}[defi]{Remarque}

\newtheorem{thm}[defi]{Th\'eor\`eme}

\newtheorem{cor}[defi]{Corollaire}

\newtheorem{lem}[defi]{Lemme}

\newtheorem{prop}[defi]{Proposition}

\begin{document}

%%%%%%%%%%%%%%%%%%%%%%%%%%%%%%%
% Add the title to the document
%%%%%%%%%%%%%%%%%%%%%%%%%%%%%%%

\maketitle

%\contribution{}

%%%%%%%%%%%%%%%%%%%%%
% Dedication (if any)
%%%%%%%%%%%%%%%%%%%%%
%\dedication{}

%%%%%%%%%%%%%%%%%%%%%%%%%%%%%%%%%%%%%%%%%%%%%%%%%%%%%%%%%%
% Add abstract, Keywords, MSC classification (recommended)
% Never remove prelims section, make it rather empty
%%%%%%%%%%%%%%%%%%%%%%%%%%%%%%%%%%%%%%%%%%%%%%%%%%%%%%%%%%
\begin{prelims}

\vspace{-0.55cm}

\def\abstractname{R\'esum\'e}
\abstract{Soit $k$ un corps de caract\'eristique z\'ero. Soit $X$ une $k$-surface projective et lisse g\'eom\'etriquement rationnelle.
Soit  $\CT$ un torseur universel sur $X$ poss\'edant un $k$-point, et soit $\CT^{c}$ une compactification lisse de $\CT$. 
La question de savoir si $\CT^{c}$ est $k$-birationnel \`a un espace projectif est encore ouverte.
On sait que les deux premiers groupes de cohomologie non ramifi\'ee de $\CT$ et $\CT^{c}$
sont r\'eduits \`a leur partie constante. On donne une condition suffisante en termes de la structure galoisienne
du groupe de Picard g\'eom\'etrique  de $X$ assurant l'\'enonc\'e analogue pour les troisi\`emes groupes de cohomologie
non ramifi\'ee de $\CT$ et $\CT^{c}$. Ceci permet de montrer que $H^{3}_{\nr}(\CT^{c},\BQ/\BZ(2))/H^3(k,\BQ/\BZ(2))$ est nul si $X$ est une surface de Ch\^atelet g\'en\'eralis\'ee, et que ce  groupe est r\'eduit \`a sa partie $2$-primaire si $X$ est une surface de del Pezzo de degr\'e au moins 2.}

\vspace{0.15cm}

\languagesection{English}{%

%\vspace{-0.05cm}
\textbf{Title. Third unramified cohomology group of universal torsors on rational surfaces} \commentskip \textbf{Abstract.} Let $k$ a field of characteristic zero. Let $X$ be a   smooth, projective, geometrically rational $k$-surface.
Let $\mathcal{T}$ be a universal torsor over $X$ with a $k$-point et $\mathcal{T}^c$ a smooth compactification of $\mathcal{T}$.
There is an open question: is $\mathcal{T}^c$ $k$-birationally equivalent to a projective space?
We know that the unramified cohomology groups of degree 1 and 2 of $\mathcal{T}$ and $\mathcal{T}^c$ are reduced to their constant part.
For the analogue of the third cohomology groups, we give a sufficient condition using the Galois structure of the geometrical Picard group of $X$.
This enables us to show that $H^{3}_{\nr}(\mathcal{T}^{c},\mathbb{Q}/\mathbb{Z}(2))/H^3(k,\mathbb{Q}/\mathbb{Z}(2))$ vanishes if $X$ is a generalised Ch\^atelet surface and that this group is reduced to its $2$-primary part if $X$ is a del Pezzo surface of degree at least 2.}
 


\end{prelims}

%%%%%%%%%%%%%%%%%%%%%
% Content begins here
%%%%%%%%%%%%%%%%%%%%%

\newpage

% Add table of contents (optional)
\setcounter{tocdepth}{1} \tableofcontents

\section{Introduction}

Soit $k$ un corps de caract\'eristique 0.
 Pour une vari\'et\'e  $X$ sur $k$ et un faisceau \'etale $F$ sur $X$,  \emph{la cohomologie non ramifi\'ee} de $X$ de degr\'e $n$ est ici par d\'efinition  le groupe 
$$H_{\nr}^n(X,F):=H^0_{Zariski}(X,\CH^n(X,F)),$$
 o\`u $\CH^n(X,F)$ est le faisceau Zariski associ\'e au pr\'efaisceau $\{ U\sbt X\}\mapsto H^n_{\acute{e}t}(U,F)$.
 Ces groupes sont des invariants $k$-birationnels de $k$-vari\'et\'es int\`egres projectives et lisses (\cite[Thm. 4.1.1]{CT95}).
Si $X$ est projective, lisse et  $k$-rationnelle, par \cite[Thm. 4.1.1 et Prop. 4.1.4]{CT95}, on a  $ H^i(k,\BQ/\BZ(j))  \iso H^{i}_{\nr}(X,\BQ/\BZ(j))$ pour tous entiers $ i\in \BN, j\in \BZ$.

Soit $X$ une $k$-surface projective, lisse, g\'eom\'etriquement rationnelle.
Soit $\CT \to X$ un torseur universel sur $X$  \cite{CTS}.  
Ce torseur est une $k$-vari\'et\'e g\'eom\'etriquement rationnelle (Corollaire \ref{H0K2UD}).
En 1979, Colliot-Th\'el\`ene
et Sansuc  \cite[Question Q1, p. 227]{CTSA} ont pos\'e la question  : si $\CT$ poss\`ede un point rationnel,
la $k$-vari\'et\'e $\CT$ est-elle  une $k$-vari\'et\'e $k$-rationnelle? Cette question est toujours ouverte.

Un certain nombre d'invariants $k$-birationnels sont triviaux sur toute compactification lisse
$\CT^{c}$ de $\CT$.
 Ainsi  les applications de restriction
  $ H^i(k,\BQ/\BZ(i-1))  \to H^{i}_{\nr}(\CT^{c},\BQ/\BZ(i-1))$ 
  sont des isomorphismes pour $i=1$ et $i=2$. Le cas $i=1$ est facile.
Dans le cas $i=2$, ceci dit que l'application de restriction
$\Br(k) \to \Br(\CT^{c})$ sur les groupes de Brauer est un isomorphisme.
Pour ce r\'esultat, voir  \cite[Thm. 1]{CTS2}, \cite[Thm. 2.1.2]{CTS}, \cite[Prop. 1.8]{HS} et le Th\'eor\`eme \ref{torsoruniv} ci-dessous. 
Par ailleurs, pour $\CT$ poss\'edant un $k$-point, on sait (Proposition  \ref{chow0paring}) 
que pour tous $i\in \BN, j\in \BZ$, l'image de $H^{i}_{\nr}(X,\BQ/\BZ(j))$ 
dans $H^{i}_{\nr}(\CT,\BQ/\BZ(j))$ est r\'eduite \`a $H^{i}(k, \BQ/\BZ(j))$.

\medskip

Dans le pr\'esent article, pour $X$, $\CT $ et $\CT^{c}$ comme ci-dessus, 
 nous \'etudions les groupes   
$$H^{3}_{\nr}(\CT^{c},\BQ/\BZ(2))/H^3(k,\BQ/\BZ(2)) \subset
 H^{3}_{\nr}(\CT,\BQ/\BZ(2))/ H^3(k,\BQ/\BZ(2)).$$
Dans \cite{C2}, nous avons \'etabli que $H^{3}_{\nr}(\CT^{c},\BQ/\BZ(2))/H^3(k,\BQ/\BZ(2))$  est fini.

\smallskip

 Les principaux r\'esultats du pr\'esent article sont les suivants.

\begin{itemize}
\item[\rm (a)] Le quotient $H^{3}_{\nr}(\CT,\BQ/\BZ(2))/ H^3(k,\BQ/\BZ(2))$
  est   fini.
  
\item[\rm (b)]  Si   $H^1(k,\Sym^2\Pic(X_{\bk}))=0$ et $X(k)\neq \emptyset$, alors $\frac{H_{\nr}^3(\CT,\BQ/\BZ(2))}{H^3(k,\BQ/\BZ(2))}=0$ (Th\'eor\`eme \ref{H4Z2}).
   Pour \'etablir ce r\'esultat, nous appliquons aux torseurs sous un tore
  une technique d\'evelopp\'ee par A.~Merkurjev \cite{Mer} pour \'etudier
   les torseurs sous un 
 groupe semisimple.  
  
\item[\rm (c)] Si $X$ est une surface de Ch\^atelet g\'en\'eralis\'ee, c'est-\`a-dire
un fibr\'e en coniques sur $\BP^1_{k}$ poss\'edant une section sur
une extension quadratique de $k$, et $X(k) \neq \emptyset$,
 alors $\frac{H_{\nr}^3(\CT^c,\BQ/\BZ(2))}{H^3(k,\BQ/\BZ(2))} $
est nul (Th\'eor\`eme \ref{coniccoh}).

\item[\rm (d)] Si $X$ est une surface projective, lisse, 
$k$-birationnellement \'equivalente \`a une surface de del Pezzo de degr\'e $\geq 2$ 
ou \`a une surface fibr\'ee en coniques au-dessus d'une conique,
 alors $\frac{H_{\nr}^3(\CT^c,\BQ/\BZ(2))}{H^3(k,\BQ/\BZ(2))} $
est purement 2-primaire
(Th\'eor\`eme \ref{sauf2prime}).

 \end{itemize}

\subsection*{Conventions et notations}

Soit $k$ un corps quelconque de caract\'eristique $0$. On note $\overline{k}$ une cl\^oture alg\'ebrique et $\Gamma_k:=\Gal(\bk/k)$.

Une $k$-vari\'et\'e $X$ est un $k$-sch\'ema s\'epar\'e de type fini. 
 Pour $X$ une telle vari\'et\'e, on note $k[X]$ son anneau des fonctions globales,
$k[X]^{\times}$ son groupe des fonctions inversibles,
et $\Pic(X):=H^1_{\text{\'et}}(X,\BG_m)$ son groupe de Picard.
Si $X$ est lisse, on note
$\Br(X):=H_{\text {\'et}}^2 (X, \BG_m)$ son groupe de Brauer,
$\CCH^i(X)$ son groupe de Chow de codimension $i$ et $\CCH_i(X)$ son groupe de Chow de dimension $i$.
Pour $X/k$ projective lisse, notons $A_0(X)\sbt \CCH_0(X)$ le groupe de classes des 0-cycles de degr\'e $0$.
Pour tous $i \in \BN, j\in \BZ$, on note $\frac{H^i_{\nr}(X,\BQ/\BZ(j))}{H^i(k,\BQ/\BZ(j))}$ le conoyau du morphisme
$H^i(k,\BQ/\BZ(j))\to H^i_{\nr}(X,\BQ/\BZ(j))$.

Tous les groupes de cohomologie ou d'hypercohomologie  utilis\'es dans cet article sont des groupes de cohomologie \'etale, 
sauf les groupes de cohomologie de Zariski $H^i(X,\CK_j)$ \`a valeurs 
dans des faisceaux de $K$-th\'eorie alg\'ebrique.
Pour chaque sch\'ema $X$, notons $ D^+_{\acute{e}t}(X)$ la cat\'egorie d\'eriv\'ee  born\'ee \`a gauche
de la cat\'egorie des faisceaux \'etales. 
Pour les propri\'et\'es de la cat\'egorie d\'eriv\'ee d'une cat\'egorie ab\'elienne, voir \cite[\S 13.1]{KS}.
Pour tout $n\in \BZ$, on a la sous-cat\'egorie $D^{\geq n}_{\acute{e}t}(X)\sbt D^+_{\acute{e}t}(X)$ (cf. \cite[Notation 13.1.11]{KS}) 
et les foncteurs de troncature $\tau^{\leq n}$ et $\tau^{\geq n}$  (\cite[D\'ef. 12.3.1 et Prop. 13.1.5]{KS}).
Soit 
$$-\otimes^L-:  D^+_{\acute{e}t}(X)\times D^+_{\acute{e}t}(X)\to D^+_{\acute{e}t}(X) : (F,G)\mapsto F\otimes_{\BZ}^LG .$$
Puisque la tor-dimension de $\BZ$ est $1$, le foncteur ci-dessus est bien d\'efini et on a
\begin{equation}\label{0opluse1}
-\otimes^L-  :D^{\geq i}_{\acute{e}t}(X)\times D^{\geq j}_{\acute{e}t}(X)\to D^{\geq i+j-1}_{\acute{e}t}(X).
\end{equation}

Soit $T$ un $k$-tore.
Notons $T^*=\Hom_{{\overline k}-gp}  (T_{\overline{k}}, \BG_{m, \overline{k}})$ le r\'eseau galoisien d\'efini par le groupe
 des caract\`eres g\'eom\'etriques du tore $T$. Donc $T^*\sbt \bk[T]^{\times}$.

Pour un groupe ab\'elien $A$ et un entier $n\in \BZ$, 
on note $A[n]:=\{x \in A, nx=0\}$ et $A_{tors}$ le sous-groupe de torsion de $A$.
On note $\Sym^2A$ la deuxi\`eme puissance sym\'etrique de $A$ (\cite[III \S 6]{Bour}) et $\wedge^i A$ la $i$-i\`eme puissance ext\'erieure de $A$ (\cite[III \S 7]{Bour}).
Si $A\cong A_1\oplus A_2$, on a des isomorphismes naturels
\begin{equation}\label{symoplus}
\Sym^2A_1\oplus (A_1\otimes A_2)\oplus \Sym^2A_2   \iso \Sym^2A  
\ \ \ \text{et}\ \ \ \      
 \wedge^2A_1\oplus (A_1\otimes A_2)\oplus \wedge^2A_2 \iso \wedge^2A 
\end{equation}
Si $A$ est un r\'eseau, on a la suite exacte naturelle
$$0 \to   \wedge^2A\to A\otimes_{\BZ}A\to \Sym^2A\to 0$$
o\`u $ \wedge^2A\to A\otimes_{\BZ}A$ envoie $a \wedge b$ sur $a\otimes b - b\otimes a$.

\section{Rappels}

Dans cette section, on fait des rappels sur plusieurs sujets: cohomologie motivique, vari\'et\'es toriques, vari\'et\'es cellulaires,  torseurs universels et leur cohomologie
\`a coefficients $\BG_{m}$.
Soit $k$ un corps de caract\'eristique~0. 

\subsection{Cohomologie motivique}

Notons $K_i(A)$ le $i$-i\`eme groupe de  $K$-th\'eorie de Quillen d'un anneau $A$.
Sur tout sch\'ema $X$, on note $\CK_i$ le faisceau pour la topologie de Zariski sur $X$ associ\'e au pr\'efaisceau 
$U\mapsto K_i(H^0(U,\CO_X)) $ et $\CK_{i,\et}$ celui pour la topologie \'etale.
On note $H^i(X,\CK_j)$ ses groupes de  cohomologie pour la topologie de  Zariski. 
Par la r\'esolution de Gersten (un th\'eor\`eme de Quillen, cf.  \cite{CTHK} pour une d\'emonstration du cas g\'en\'eral), 
si $X$ est une $k$-vari\'et\'e lisse, le groupe  $H^i(X,\CK_j)$ est le $i$-i\`eme groupe de cohomologie du complexe :
\begin{equation}\label{resolutiongersten}
0\to \oplus_{x\in X^{(0)}}K_j(k(x))\to \oplus_{x\in X^{(1)}}K_{j-1}(k(x))\to \cdots \to \oplus_{x\in X^{(j)}}\BZ \to 0,
\end{equation}
o\`u $X^{(l)}$ est l'ensemble des points de codimension $l$ de $X$ pour tout $l\in \BZ_{\geq 0}$.

On utilise le complexe motivique $\BZ(r)$ de faisceaux de cohomologie \'etale sur les vari\'et\'es lisses sur $k$ pour $r=0,1,2$, comme d\'efini par Lichtenbaum (\cite{L1} \cite{L2}).
On utilise le complexe motivique ``de Lichtenbaum'' pour pouvoir utiliser la m\'ethode de  Merkurjev dans \cite{Mer}.
 Alors $\BZ(0)=\BZ$, $\BZ(1)\iso \BG_m[-1]$, et $\BZ(2)$ est support\'e en degr\'e 1 et 2. 
Si $X$ est un sch\'ema de type fini sur $k$, on a $\CH^2(X,\BZ(2))\cong \CK_{2,\et}(X)$ et
un triangle dans $D^+_{\acute{e}t}(\CT)$ pour $r=0,1,2$:
\begin{equation}\label{triangleZ2}
\xymatrix{\BZ(r)\ar[r]^n &\BZ(r)\ar[r]&\BZ/n(r)\ar[r]^-{+1}&.}
\end{equation}
 
 De plus, on sait (\cite{K93,K96,Sus}, cf. \cite[\S 1]{CT1}):

\begin{thm}\label{motivic}
Soit $X$ une  $k$-vari\'et\'e lisse g\'eom\'etriquement int\`egre de corps de fonctions $k(X)$.
\begin{itemize}
\item[\rm (i)] On a les \'egalit\'es : $\BH^0(X,\BZ(2))=0$, $\BH^1(X,\BZ(2))=K_{3,indec}(k(X))$, $\BH^2(X,\BZ(2))=H^0(X,\CK_2)$ et $\BH^3(X,\BZ(2))=H^1(X,\CK_2)$, o\`u 
$$K_{3,indec}(k(X)):=\Coker(K_3^{Milnor}(k(X))\to K_3(k(X))).$$
De plus, $K_{3,indec}(\bk(X))$ est extension d'un groupe uniquement divisible par $\BQ/\BZ(2)$.
\item[\rm (ii)]
On a  une suite exacte naturelle:
\begin{equation}\label{e1}
\xymatrix{0\ar[r]&\CCH^2(X)\ar[r]^-{\mathrm{cl}}&\BH^4(X,\BZ(2))\ar[r]&H^3_{\nr}(X,\BQ/\BZ(2))\ar[r]&0,}
\end{equation}
o\`u $\mathrm{cl}$ est induit par $\CCH^2(X)\cong H^4(X,\CK_2)\cong \BH^4_{\mathrm{Zari}}(X,\BZ(2))\to \BH^4(X,\BZ(2))$ {\rm (\cite[(10)]{K96})}.
\end{itemize}
\end{thm}

Dans \cite[Prop. 2.5]{L1}, Lichtenbaum a d\'efini le cup-produit des complexes motiviques 
$$\cup:\  \BZ(1)\otimes^L \BZ(1)\to \BZ(2),$$
et ce morphisme induit un morphisme de faisceaux:
\begin{equation}\label{motivice1}
\BG_m\otimes \BG_m\cong \CH^2(\BZ(1)\otimes^L \BZ(1))\to \CH^2(\BZ(2))\cong \CK_{2,\et},
\end{equation}
qui est exactement le morphisme $\BG_m\otimes \BG_m\cong \CK_1\otimes \CK_1\to \CK_{2,\et} $ induit par le cup-produit de groupes de K-th\'eorie $K_1\otimes K_1\to K_2$ (\cite[Rem. 2.6]{L1}).
Donc la composition
\begin{equation}\label{motivice2}
 H^2(X,\BZ_X(1))\otimes H^2(X,\BZ_X(1))\iso \Pic(X)\otimes \Pic(X)\xrightarrow{\mathrm{intersection}} \CCH^2(X)\to \BH^4(X,\BZ(2))
 \end{equation}
est exactement le morphisme induit par le cup-produit ci-dessus.

\subsection{Vari\'et\'es toriques}
Par \cite[Cor. 2]{Su} et \cite[Thm. 1.10]{Oda}  on a:

\begin{thm}\label{torique}
Supposons que $k$ est alg\'ebriquement clos. Soit $X$ une vari\'et\'e torique lisse sous le tore $\BG_m^n$ sur $k$.
Alors pour chaque $\BG_m^n$-orbite $Z$ de codimension $i$, il existe une sous-vari\'et\'e torique ouverte $U\sbt X$ telle que $Z\sbt U$ et $U\iso \BG_m^{n-i}\times \BA^{i}$ comme vari\'et\'e torique, o\`u l'action de $\BG_m^n$ sur $\BG_m^{n-i}\times \BA^{i}$ est la multiplication.
\end{thm}

Comme   cons\'equence, on a:

\begin{cor} Sous les hypoth\`eses du Th\'eor\`eme  \ref{torique},
soient $Z_1$, $Z_2$ deux $\BG_m^n$-orbites de codimension $1$ de $X$. Alors il existe un nombre fini de $\BG_m^n$-orbites de codimension $2$, notons-les $S_j$, telles que
leurs adh\'erences sch\'ematiques satisfont
 $\overline{Z_1}\cap\overline{Z_2}=\bigcup_j\overline{S_j}$.
\end{cor}

\subsection{Vari\'et\'es cellulaires}
 
\begin{defi}[{\cite[D\'efinition 3.2]{K97}}]{\rm
Une vari\'et\'e $X$ sur $k$ a une d\'ecomposition cellulaire (bri\`evement: est cellulaire) si elle est vide ou s'il existe un sous-ensemble ferm\'e propre $Z\sbt X$ tel que $X\setminus Z$ soit isomorphe \`a un espace affine et $Z$ ait une d\'ecomposition cellulaire.}
\end{defi}

\begin{prop}[{\cite[Prop. 2.2]{C2}}] \label{exam1}
 Soit $k$  un corps alg\'ebriquement clos.
\begin{itemize}
\item[\rm (1)] Une surface projective, lisse, $k$-rationnelle est cellulaire.
\item[\rm (2)] (\cite[Lemme, p. 103]{Ful93}) Une vari\'et\'e torique, projective, lisse sur $k$ est cellulaire.
\item[\rm (3)] Soient $T$ un tore sur $k$ et $T^c$ une $T$-vari\'et\'e torique,  projective, lisse.
Soient $X$ une vari\'et\'e cellulaire sur $k$ et  $Y\ra X$ un $T$-torseur.
 Alors $Y^c:=Y\times^TT^c$ est cellulaire.
\end{itemize}
\end{prop}

Rappelons que $Y\times^TT^c$ est le quotient $(Y\times T^c)/T$, 
o\`u l'action de $T$ est d\'efini par $t\cdot (y,a):= (t\cdot y, t^{-1}\cdot a)$
pour tous $t\in T$, $y\in Y$ et $a\in T^c$.
Ce quotient existe par \cite[Thm. 4.19]{PV} et \cite[Appendice 1, Thm. 6]{An73}.

Par \cite[Exemple 19.1.11]{Ful}, on a

\begin{thm}\label{cellulairecoho}
Supposons que $k$ est alg\'ebriquement clos. Soient $X$ une vari\'et\'e lisse cellulaire sur $k$ et $n$ un entier.
Pour tout entier $i$, le groupe $\CCH^i(X)$ est de type fini et sans torsion, et le
morphisme classe de cycle $\CCH^i(X)\otimes \BZ/n\ra H^{2i}(X,\BZ/n)$ est un isomorphisme.
Pour tout entier $i$ impair, on  a $H^i(X,\BZ/n)=0 $ et donc $H^{i}(X,\BZ_{l})=0$ pour tout nombre premier $l$.
\end{thm}

Soit $X$ une vari\'e\'te lisse. On a un homomorphisme naturel:
\begin{equation}\label{H1K2controle}
\Pic (X)\otimes k^{\times}\to H^1(X,\CK_2).
\end{equation}
En effet, on a le diagramme commutatif suivant
$$ \xymatrix{k(X)^{\times}\otimes k^{\times}\ar[r]\ar[d] &\oplus_{x\in X^{(1)}}k^{\times}\ar[r]\ar[d]^{\phi_0}& \Pic(X)\otimes k^{\times} \ar[r]&0\\
K_2k(X)\ar[r]^{d_2} & \oplus_{x\in X^{(1)}}k(x)^{\times}\ar[r]^{d_1} & \oplus_{x\in X^{(2)}}\BZ&,
}  $$
o\`u la premi\`ere ligne est obtenue \`a partir de la suite exacte d\'efinissant le groupe $\Pic(X)$ par tensorisation avec $k^{\times}$ et la deuxi\`ere ligne est la r\'esolution de Gersten.
On v\'erifie que la compos\'e $d_1\circ \phi_0$ vaut z\'ero, ce qui permet de d\'efinir l'homomorphisme (\ref{H1K2controle}) par chasse au diagramme.

\begin{prop}\label{H1K2control}
Supposons que $k$ est alg\'ebriquement clos. Soit $X$ une vari\'et\'e lisse, connexe, rationnelle sur $k$. 
Supposons que $X$ est projective ou cellulaire.
Alors l'homomorphisme (\ref{H1K2controle}) est un isomorphisme.
\end{prop}

\begin{proof}
Si $X$ est projective, ceci r\'esulte de A. Pirutka \cite[Prop. 2.6]{Pi}.

Si $X$ est cellulaire, on fixe une d\'ecomposition cellulaire de $X$.
Soient $n=\dim (X)$ et $U\sbt X$ l'ouvert isomorphe \`a $\BA^n$ dans la d\'ecomposition cellulaire.
Notons $Z:=X\setminus U$.
On a  $H^0(U,\CK_2)=K_2(k)$, $H^1(U,\CK_2)=0$ et $\Pic(X)\cong \Div_{Z}(X)$.
Par la r\'esolution de Gersten (\ref{resolutiongersten}), on a une suite exacte:
$$0\to H^0(X,\CK_2)\to H^0(U,\CK_2) \to \ker (\psi)\xrightarrow{\phi_1} H^1(X,\CK_2)\to H^1(U,\CK_2)=0,$$
o\`u 
$\psi:  \oplus_{x\in Z^{(0)} }k(x)^{\times}\xrightarrow{div} \oplus_{x\in Z^{(1)}}\BZ .$
Alors $H^0(X,\CK_2)\cong H^0(U,\CK_2)  $ et $\phi_1$ est un isomorphisme.

Puisque $div(k^\times)=0$, on a $\oplus_{x\in Z^{(0)}}k^{\times}\sbt \ker (\psi)$.
Pour tout $x\in Z^{(0)}$, il existe une unique sous-vari\'et\'e r\'eduite localement ferm\'ee $V_x$ de codimension $1$ dans la d\'ecomposition cellulaire
telle que $x\in V_x$, $V_x\cong \BA^{n-1}$ et $V_x\cap V_{x'}=\emptyset$ pour $x\neq x'\in Z^{(0)}$. 
Donc 
$$k^{\times}\cong \ker (k(x)^{\times}\xrightarrow{div} \oplus_{y\in V_x^{(1)}}\BZ)\ \ \ \text{et}\ \ \  \ker (\psi)\sbt \oplus_{x\in Z^{(0)}}k^{\times}.$$
Alors $\ker (\psi)\cong \oplus_{x\in Z^{(0)}}k^{\times} \cong \Div_Z(X)\otimes  k^{\times}\cong  \Pic(X)\otimes  k^{\times}  $.
Puisque $\phi_0$ et $\phi_1$ sont induits par les inclusions $\{k^{\times}\sbt k(x)^{\times}\}_{x\in X^{(1)}}$, on a $\phi_0=\phi_1$ et le r\'esultat en d\'ecoule.
\qed
\end{proof}

\subsection{Torseurs universels}

 Soient $X$ une $k$-vari\'et\'e lisse, $T$ un $k$-tore et $\CT\to X$ un $T$-torseur.
 La composition 
$H^1(X,T) \to H^1(X_{\bk},T_{\bk}) \to Hom(T^*, \Pic(X_{\bk}))$
associe \`a   $\CT\to X$ un homomorphisme galoisien $T^*\xrightarrow{type}\Pic(X_{\bk})$,
appel\'e \emph{le type du torseur}. Lorsque le type est un isomorphisme,
on dit \cite{CTS} que $\CT \to X$ est \emph{un torseur universel} sur $X$.

\begin{thm}\label{torsoruniv}
Soit $X$ une $k$-vari\'et\'e lisse g\'eom\'e\-triquement int\`egre, avec $\bk^{\times}  \cong  \bk[X]^{\times} $ et $\Pic (X_{\bk})$ de type fini et sans torsion. 
Soit $\CT \ra X$ un  torseur universel. 
Soit $T^c$ une $T$-vari\'et\'e torique, projective,  lisse.
Soit  $\CT^{c}= \CT\times^TT^c$.

On a les propri\'et\'es suivantes :
\begin{itemize}
\item[\rm (i)] {\rm \cite[Thm. 2.1.2]{CTS}} $ \bk^{\times}  \cong\bk[\CT]^{\times}$.

\item[\rm (ii)] {\rm \cite[Thm. 2.1.2]{CTS}}  $\Pic (\CT_{\bk})\cong 0$.

\item[\rm (iii)] {\rm \cite[Thm. 2.1.2]{CTS}}  On a un isomorphisme de $\Gamma_k$-modules  $\ \Div_{\CT^c-\CT} (\CT^c_{\bk})  \cong  \Pic (\CT^c_{\bk})$.
En particulier $\Pic (\CT^c_{\bk})$ est un module de permutation de type fini.
 
\item[\rm (iv)] {\rm \cite[Thm. 1.6]{HS}}  L'application naturelle  $\Br(X_{\bk}) \to \Br(\CT_{\bk})$ est un isomorphisme. 
 
\item[\rm (v)] {\rm \cite[Rem. 2.8.4]{CTS}} Pour toute extension $K/k$ de corps, si $X_K$ est projective et $K$-rationnelle, alors  $\CT_K$ est stablement $K$-rationnelle. 
  Si de plus $K={\overline k}$, alors $\CT_K$ est rationnelle.
\end{itemize}
\end{thm}

Par la suite spectrale de Hochschild-Serre, on a: $k^{\times}\cong k[\CT]^{\times}$, $\Pic(\CT)=0$ et $\Br(k)\cong \Br_1(\CT)$,
o\`u $\Br_1(\CT):=\Ker(\Br(\CT)\to \Br(\CT_{\bk}))$.
De plus, si $X_{\bk}$ est $\bk$-rationnelle, par le Corollaire \ref{H0K2UD} ci-dessous, on a $\Br(k)\cong \Br(\CT) $.

\begin{cor}\label{H0K2UD}
Sous les hypoth\`eses du Th\'eor\`eme \ref{torsoruniv}, supposons que $k$ est alg\'e\-briquement clos et $X$ est projective et rationnelle.
Alors $\CT$ est rationnelle, $ \Br(\CT)=0$, $H^1(\CT,\BZ/n)=0$, $H^2(\CT,\BZ/n)= 0$ et le groupe $H^0(\CT,\CK_2)$ est uniquement divisible.
\end{cor}

\begin{proof}
Dans ce cas, $T\cong \BG_m^r$ pour un certain $r\in \BZ_{>0}$. 
Puisque $H^1_{\mathrm{Zari}}(-,\BG_m)=H^1(-,\BG_m)$, il existe un ouvert $U\sbt X$ tel que $\CT\times_XU\cong U\times \BG_m^r$. 
Donc $\CT$ est rationnelle.

Par le Th\'eor\`eme \ref{torsoruniv} (iv),  $ \Br(\CT)=0$.
Par la suite de Kummer, 
et le Th\'eor\`eme \ref{torsoruniv} (i) et (ii), ceci implique
$H^1(\CT,\BZ/n)=0$ et $H^2(\CT,\BZ/n)= 0$.
Par le Th\'eor\`eme \ref{motivic} et le triangle (\ref{triangleZ2}) sur $\BZ(2)$,
on a $0\cong H^1(\CT,\BZ/n(2))\twoheadrightarrow H^0(\CT,\CK_2)[n]$ et $H^0(\CT,\CK_2)/n\hookrightarrow H^2(\CT,\BZ/n(2))\cong 0$.
Ainsi $H^0(\CT,\CK_2)$ est uniquement divisible.
\qed
\end{proof}

Comme cons\'equence, on a $\bk[\CT^c]^\times /\bk^\times =0$, $\Br(\CT^c_{\bk})=0$, et $H^1(\CT^c_{\bk},\BZ/n)= 0$, mais ceci r\'esulte d\'ej\`a de la rationalit\'e de la
$\bk$-vari\'et\'e projective et lisse $\CT^c_{\bk}$.

\begin{cor}\label{zerorationnel}
Sous les hypoth\`eses du Th\'eor\`eme \ref{torsoruniv}, supposons qu'il existe une extension finie $K/k$ de corps de degr\'e $d$
 telle que la vari\'et\'e $X_K$ soit projective et $K$-rationnelle.
 Alors, pour tous $i\in \BN,j\in \BZ$, le groupe $\frac{H^i_{\nr}(\CT^c,\BQ/\BZ(j))}{H^i(k,\BQ/\BZ(j))}$ est annul\'e par $d$.
\end{cor}

\begin{proof}
D'apr\`es le Th\'eor\`eme \ref{torsoruniv}\,(v), la\,$K$-vari\'et\'e\,$\CT_K$\,est stablement\,$K$-rationnelle. 
Ainsi\,$\frac{H^i_{\nr}(\CT^c_K,\BQ/\BZ(j))}{H^i(K,\BQ/\BZ(j))}$ $=0$ pour tous $i,j$.
Puisque le transfert est bien d\'efini pour $H^i_{\nr}(-,\BQ/\BZ(2))$ et $H^i(-,\BQ/\BZ(2))$, on peut d\'efinir le transfert
$\frac{H^i_{\nr}(\CT^c_K,\BQ/\BZ(j))}{H^i(K,\BQ/\BZ(j))}\xrightarrow{tr}\frac{H^i_{\nr}(\CT^c,\BQ/\BZ(j))}{H^i(k,\BQ/\BZ(j))} $.
Un argument de restriction-inflation 
donne le r\'esultat.
\qed
\end{proof}

En utilisant la Proposition \ref{exam1}(i), dans un pr\'ec\'edent article 
j'ai \'etabli :

   \begin{thm}[{\cite[Thm. 2.7]{C2}}]  \label{thmdecao}
Soit $X$ une $k$-surface projective, lisse, g\'eom\'e\-triquement rationnelle.
Soit $\mathcal{T} \to X$ un torseur universel sur $X$ et soit $\mathcal{T}^c$
une $k$-compactification lisse de  $\mathcal{T}$. Alors  le groupe
$\frac{H^3_{\nr}(\mathcal{T}^c,\BQ/\BZ(2))}{H^3(k, \BQ/\BZ(2))}$ est fini.
\end{thm}

\subsection{Accouplements avec les 0-cycles}

Soit $X$ une vari\'et\'e projective, lisse, g\'eom\'e\-triquement int\`egre. 
Pour toute extension de corps $K/k$ et tous $i \in \BN,j\in \BZ$,
on a un accouplement naturel :
$$A_0(X_K)\times H^i_{\nr}(X,\BQ/\BZ(j))\to H^i(K,\BQ/\BZ(j)):\ (\sum_in_iP_i, \alpha)\mapsto \sum_i \mathrm{Norm}_{K(P_i)/K}P_i^*(\alpha),$$
o\`u $P_i: \Spec\ K(P_i)\to X$ (voir \cite[formule (3)]{M3}).

\begin{prop}\label{chow0paring}
Soit $X$ une $k$-surface projective, lisse, g\'eom\'etriquement rationnelle.
Soit $\mathcal{T} \to X$ un torseur universel sur $X$ et soit $\mathcal{T}^c$
une $k$-compactification lisse de  $\mathcal{T}$.
Supposons que $\CT^c(k)\neq\emptyset$.
Alors l'homomorphisme 
$$H^i_{\nr}(X,\BQ/\BZ(j))\to \frac{H^i_{\nr}(\CT^c,\BQ/\BZ(j))}{H^i(k,\BQ/\BZ(j))}$$
 est nul pour tous $i \in \BN,j\in \BZ$.
\end{prop}

\begin{proof}
Soit $T$ un tore tel que $T^*\cong \Pic(X_{\bk})$.
Alors $\CT\to X$ est un $T$-torseur, i.e. $[\CT]\in H^1(X,T)$.

Il existe  une $T$-vari\'et\'e torique lisse $T^c$ (cf. \cite[Cor. 1]{CTHS}).
Puisque le groupe $H^i_{\nr}(\bullet,\BQ/\BZ(j))$ est un invariant birationnel des vari\'et\'es projectives et lisses,
et que l'existence d'un $k$-point est un invariant $k$-birationnel des $k$-vari\'et\'es projectives et lisses (lemme de Nishimura),
 on peut supposer que $\CT^c\cong T^c\times^T\CT$.
Il  existe alors un morphisme $\CT^c\xrightarrow{\pi} X$ \'etendant $\CT \to X$.  
 
Par \cite[\S II.B]{CTSA},  pour toute extension $K/k$ de corps, $[\CT]$ induit un homomorphisme 
$$A_0(X_K)\xrightarrow{\theta} H^1(K,T): \sum_ix_i\mapsto \sum_i\Res_{K(x_i)/K}x_i^*[\CT]$$
Puisque $[\CT]|_{\CT}=0\in H^1(\CT,T)$ et  que tout \'el\'ement de $A_0(\CT^c_K)$ \'equivaut \`a un \'el\'ement support\'e sur $\CT_K$
(lemme de d\'eplacement facile sur les z\'ero-cycles sur une vari\'et\'e lisse),
la composition des homomorphismes
$$A_0(\CT^c_K)\xrightarrow{\pi_*}A_0(X_K)\xrightarrow{\theta}H^1(K,T): \sum_it_i\mapsto \sum_i\Res_{K(t_i)/K}t_i^*(\pi^*[\CT])$$
est nulle.

D'apr\`es  \cite[Prop. 4]{CT83} et \cite[Thm. 3]{CTS81}, pour toute extension\,$K/k$\,de corps, 
le morphisme\,$A_0(X_K)\xrightarrow{\theta} H^1(K,T) $ est injectif.
Ainsi le morphisme $A_0(\CT^c_K)\xrightarrow{\pi_*}A_0(X_K)$ est nul.

Pour toute extension $K/k$ de corps, on a des accouplements compatibles:
$$\xymatrix{A_0(\CT^c_K)\ar[d]^{\pi_*}\ar@{}[r]|-{\times}&H^i_{\nr}(\CT^c,\BQ/\BZ(j))\ar[r]^{(,)_{\CT^c}}&H^i(K,\BQ/\BZ(j))\ar[d]^=\\
A_0(X_K)\ar@{}[r]|-{\times}&H^i_{\nr}(X,\BQ/\BZ(j))\ar[r]^{(,)_X}\ar[u]^{\pi^*}&H^i(K,\BQ/\BZ(j)).
}$$
Soit $t\in \CT^c(k)$, $K:=k(\CT)$ et $\eta\in \CT$ le point g\'en\'erique.
Alors $t-\eta\in A_0(\CT^c_K)$ et, pour tout $\alpha\in H^i_{\nr}(X,\BQ/\BZ(j))$, on a 
$$(t,\pi^*\alpha)_{\CT^c}-(\eta, \pi^*\alpha)_{\CT^c}=(t-\eta, \pi^*\alpha)_{\CT^c}=(\pi_*(t-\eta ),\alpha)_X=(0,\alpha)_X=0\in H^i(K,\BQ/\BZ(j)).$$ 
Puisque  $(\eta,-)_{\CT^c}$ est l'inclusion canonique $H^i_{\nr}(\CT^c,\BQ/\BZ(j))\sbt H^i(k(\CT^c),\BQ/\BZ(j))$
et
$$\Im (t,-)_{\CT^c}\sbt \Im (H^i(k,\BQ/\BZ(j))\to H^i(K,\BQ/\BZ(j))),$$
on a $\Im (\pi^*)\in \Im (H^i(k,\BQ/\BZ(j))\to H^i_{\nr}(\CT^c,\BQ/\BZ(j)) ) $.
\qed
\end{proof}

\section{Cohomologie motivique \`a coefficients $\BZ(2)$ d'un torseur sous un tore}

Dans  \cite{Mer}, A. Merkurjev a \'etudi\'e la cohomologie motivique \`a coefficients $\BZ(2)$ pour un torseur
sous un groupe semisimple. Nous reprenons sa m\'ethode pour \'etudier
les torseurs sous un tore.
Soient $T$ un tore sur $k$ de dimension $N$, et $f:Y\ra X$ un torseur sous $T$ sur $X$, o\`u $X$ est une $k$-vari\'et\'e lisse g\'eom\'etriquement int\`egre.
 On calcule la relation entre la cohomologie motivique de $X$ et celle  de $Y$ \`a coefficients dans le complexe $\BZ(2)$, 
 i.e. on calcule la cohomologie de $\BZ_f(2)$ ci-dessous (Th\'eor\`eme \ref{mainz2}).

La plan de la d\'emonstraction du Th\'eor\`eme \ref{mainz2} est:
 (i) on calcule les faisceaux cohomologiques de $\BZ_f(2)$, en utilisant une m\'ethode de Merkurjev (\cite{Mer}), qui g\'en\'eralise le travail de Sansuc sur $\BZ_f(1)$;
 (ii) on  calcule la cohomologie de $\BZ_f(2)$ sur $X_{\bk}$.
 Ce th\'eor\`eme sera appliqu\'e \`a calculer la cohomologie de $\BZ_f(2)$ sur $X$ (Th\'eor\`eme \ref{H4Z2}) en utilisant la suite spectrale de Hochschild-Serre.

Dans ce cas,  on d\'efinit $\BZ_f(i)$ le c\^one du morphisme naturel  $\BZ_X(i)\to Rf_*\BZ_Y(i)$ pour $i=1,2$ (\cite[(4.3)]{Mer}),
i.e. on a deux triangles dans $D^+_{\acute{e}t}(X)$:
\begin{equation}\label{Zde1}
\BZ_X(1)\to Rf_*\BZ_Y(1)\to \BZ_f(1)\ra \BZ_X(1)[1],
\end{equation}
et
\begin{equation}\label{Zde2}
\BZ_X(2)\xrightarrow{f^{\#}(2)} Rf_*\BZ_Y(2)\xrightarrow{d(2)} \BZ_f(2)\ra \BZ_X(2)[1].
\end{equation}

Pour chaque $\Gamma_k$-module $M$ continu discret,
 on peut voir $M$ comme un faisceau \'etale sur le petit site \'etale (\'et/$k$), et son image inverse sur le grand site \'etale est not\'e \'egalement par $M$. 
Ainsi  $\bigwedge^*T^*$ est un faisceau \'etale. 

Soient 
\begin{equation}
R^{\leq 2}f_*\BZ_Y(1):=\tau^{\leq 2} Rf_*\BZ_Y(1)\ \ \ \text{et} \ \ \ \BZ^{\leq 2}_f(1):=\tau^{\leq 2}\BZ_f(1).\end{equation}

\begin{prop}\label{Zf1}
On a $ \BZ_f^{\leq 2}(1)\cong T^*[-1]$ et donc  $\BZ_X(1)\otimes \BZ_f^{\leq 2}(1)\cong \BG_m\otimes T^*[-2]$ et  
$$\tau^{\leq 2}(\BZ_X(1)\otimes^L\BZ_f(1))\cong \BG_m\otimes T^*[-2].$$
\end{prop}

\begin{proof}
Puisque $\BZ_X(1)=\BG_m[-1]|_X$ et $\BZ_Y(1)=\BG_m[-1]|_Y$, 
on a: $\CH^1(\BZ_X(1))\cong \BG_m|_X$, $\CH^i(\BZ_X(1))=0$ pour tout $i\neq 1$ et $\CH^i(Rf_*\BZ_Y(1))=R^{i-1}f_*(\BG_m|_Y)$ pour tout $i$.
Ainsi $\CH^i(Rf_*\BZ_Y(1))=0$ pour $i\leq 0$ et
le faisceau $\CH^2(Rf_*\BZ_Y(1))$ est le faisceau \'etale associ\'e au pr\'efaisceau:
$$U\lto \BH^2(f^{-1}U,\BZ(1))=\Pic(f^{-1}U).$$
Localement pour la topologie \'etale,  $f^{-1}U\iso U\times \BG^N_m$ et $U$ est le spectre d'un anneau local r\'egulier. 
Dans ce cas on a $0\cong \Pic(U)\cong \Pic(f^{-1}U)$.
 Donc $\CH^2(Rf_*\BZ_Y(1))=0$.
 Ainsi
 \begin{equation}\label{eq3.1}
 \CH^i(\BZ_X(1))=0\ \ \text{pour}\ \ i\neq 1,\ \    \CH^i(Rf_*\BZ_Y(1))=0\ \ \text{pour}\ \  i\leq 0\ \ \text{et}\ \ i=2.
 \end{equation}

D'apr\`es (\ref{Zde1}), on a une suite exacte longue:
$$\cdots \to \CH^i(\BZ_X(1))\xrightarrow{\phi_i} \CH^i(Rf_*\BZ_Y(1))\to \CH^i(\BZ_f(1))\to \CH^{i+1}(\BZ_X(1))\xrightarrow{\phi_{i+1}} \cdots.$$
D'apr\`es  \cite[Prop. 1.4.2]{CTS}, on a une suite exacte:
\begin{equation}\label{e97}
0\to \BG_{m,X}\to f_*\BG_{m,Y}\to T^*\to 0
\end{equation}
et donc $\phi_1$ est injectif de conoyau $T^*$.
D'apr\`es \eqref{eq3.1}, on a $\CH^1(\BZ_f(1))\cong T^*$ et $\CH^i(\BZ_f(1))=0$ pour $i\leq 0$ et $i=2$.
Ceci donne le premier \'enonc\'e.

 Le deuxi\`eme \'enonc\'e r\'esulte du fait que: $T^*\otimes^L-\cong T^*\otimes -$.
 
Le troisi\`eme \'enonc\'e r\'esulte du fait que: d'apr\`es (\ref{0opluse1}), on a $\BG_m\otimes^L (D_{\acute{e}t}^{\geq 3}(X))\sbt D_{\acute{e}t}^{\geq 2}(X)$.
\qed
\end{proof}

La Proposition \ref{Zf1} induit un triangle dans $D^+_{\acute{e}t}(X)$:
\begin{equation}\label{Zde12}
\BZ_X(1)\xrightarrow{f^{\#}(1)} R^{\leq 2}f_*\BZ_Y(1)\xrightarrow{d(1)} \BZ^{\leq 2}_f(1)\ra \BZ_X(1)[1].
\end{equation}
et un isomorphisme:
\begin{equation}\label{Zf1-e}
\BG_m\otimes T^*\cong \CH^2(\BZ_X(1)\otimes^L\BZ^{\leq 2}_f(1))\cong \CH^2(\BZ_X(1)\otimes^L\BZ_f(1)).
\end{equation}

Si $Y\cong X\times T$, la projection $X\times T\to T$ induit une section de (\ref{Zde12}):
\begin{equation}\label{Zf1-e1}
T^*[-1]\cong \tau^{\leq 2}\BZ_f(1)\to Rf_*\BZ_Y(1).
\end{equation}

\medskip

Le r\'esultat principal de cette section est la Proposition \ref{cohzf2},
qui calcule   $\BZ_f(2)$.

\medskip

Soit $T_0:=\BG_m^N$.
Puisque $H^0(T_0,\CK_1)=k[T_0]^{\times}\supset T_0^*$,
ceci induit un homomorphisme canonique $T_0^* \to H^0(T_0,\CK_1)$.
Pour chaque vari\'et\'e lisse int\`egre $U$ et tout $n\in \BN$, le cup-produit 
$$  [H^n(U,\CK_2)\otimes H^0(T_0,\CK_0)]\oplus [H^n(U,\CK_1)\otimes H^0(T_0,\CK_1)]\xrightarrow{\cup} H^n(U\times T_0,\CK_2) $$
induit un homomorphisme:
\begin{equation}\label{Ktore-e1}
H^n(U,\CK_2)\oplus [H^n(U,\CK_1)\otimes T_0^*]\xrightarrow{\cup} H^n(U\times T_0,\CK_2). 
\end{equation}

\begin{lem}\label{Ktore}
Soit $T_0:=\BG_m^N$ et $U$ une vari\'et\'e lisse int\`egre.
Alors, pour  tout $n\in \BN$, 
on a une suite exacte canonique (o\`u $\cup$ est d\'efini dans (\ref{Ktore-e1})): 
$$0\to H^0(U,\CK_2)\oplus [H^0(U,\CK_1)\otimes T_0^*]\xrightarrow{\cup} H^0(U\times T_0,\CK_2)\to
 H^0(U,\CK_0)\otimes\wedge^2T_0^* \to 0.$$
 et un isomorphisme $H^n(U,\CK_2)\oplus [H^n(U,\CK_1)\otimes T_0^*]\xrightarrow{\cup} H^n(U\times T_0,\CK_2)$  pour $n\geq 1$.
 \end{lem}

\begin{proof}
D'apr\`es (\ref{resolutiongersten}), on a $H^n(U,\CK_0)=0$ pour $n\geq 1$.

Notons $K_i^M$ les groupes de $K$-th\'eorie alg\'ebrique de Milnor. 
Par \cite[\S 4]{Mer1}, pour toute vari\'et\'e lisse $X$, on a un complexe de Gersten: 
$$0\to \oplus_{x\in X^{(0)}}K^M_j(k(x))\to \oplus_{x\in X^{(1)}}K^M_{j-1}(k(x))\to \cdots \to \oplus_{x\in X^{(j)}}\BZ \to 0$$
et on note $A^i(X,K^M_j)$ les groupes de cohomologie de ce complexe.
Donc $A^i(X,K^M_0)=0$ pour $i\geq 1$.

Soit $f_1,\dots,f_N$ une base de $T_0^*$. Ainsi $\{f_{i_1},\dots,f_{i_q}\}_{1\leq i_1<\dots<i_q\leq N}$ est un \'el\'ement de $A^0(T_0,K_q)$.
Le $A^*(U,K_*)$-module $A^*(U\times T_0,K_*)$ est libre avec une base consistant en les \'el\'ements 
$\{f_{i_1},\dots,f_{i_q}\}$ pour $q=0,1,\cdots, N$ et tout $q$-uplet $1\leq i_1< i_2<\cdots <i_q\leq N$.
Pour $N=1$, ceci est  \cite[Proposition 5.5]{Mer1}. Le cas g\'en\'eral en d\'ecoule par r\'ecurrence
(cf. \cite[Corollaire 5.6]{Mer1}).

Pour $K_2^M$, on a 
$$A^n(U\times T_0,K_2^M)\cong A^n(U,K_2^M)\oplus [\oplus_{i}A^n(U,K_1^M)\cup \{f_i\}]\oplus
 [\oplus_{i<j}A^n(U,K_0^M)\cup \{f_i,f_j\}]. $$
Alors le cup-produit induit une injection (l'analogue de (\ref{Ktore-e1}))
$$A^n(U,K_2^M)\oplus [A^n(U,K_1^M)\otimes T_0^*]\stackrel{\cup_{1,1}}{\hookrightarrow} A^n(U\times T_0,K_2^M)$$ 
et $\coker (\cup_{1,1})\cong \oplus_{i<j}A^n(U,K_0^M)\cup \{f_i,f_j\}$.
Donc le morphisme (cf. \cite[(5.8)]{Mer1}):
$$A^n(U,K_0^M)\otimes\wedge^2T_0^*\to \coker (\cup_{1,1}): a\otimes (f_i\wedge f_j)\to a\cup \{f_i,f_j\} $$
est bien d\'efini et c'est un isomorphisme.
Ceci donne une suite exacte 
$$0\to A^n(U,K^M_2)\oplus [A^n(U,K^M_1)\otimes T_0^*]\xrightarrow{\cup} A^n(U\times T_0,K^M_2)\to A^n(U,K^M_0)\otimes\wedge^2T_0^* \to 0.$$
Donc $\cup:\  A^n(U,K^M_2)\oplus [A^n(U,K^M_1)\otimes T_0^*]\iso A^n(U\times T_0,K^M_2)$ pour $n\geq 1$.

Pour tout corps $F$ et $i=0,1,2$, on a $K_i^M(F)\iso K_i(F)$ (voir D3 dans \cite[(1.12)]{Ro}). 
D'apr\`es (\ref{resolutiongersten}), pour $j=0,1,2$, on a donc un isomorphisme canonique $A^i(X,K_j^M)\iso H^i(X,\CK_j)$.
Ceci donne le r\'esultat annonc\'e.
\qed
\end{proof}

\begin{lem}\label{BZf0}
Les faisceaux $\CH^i(\BZ_f(2))$ sont nuls pour $i\leq 1$ et $i=3$.
\end{lem}

\begin{proof}
On suit la d\'emonstration dans \cite[pp. 10]{Mer}

Puisque $\BZ_X(2)$ et  $\BZ_Y(2)$ sont support\'es en degr\'es 1 et 2, 
on a $\CH^i(\BZ_f(2))=0$ pour $i<0$,  $R^{i}f_*\BZ_Y(2)=0$ pour $i\leq 0$, $R^if_*\BZ_Y(2)\cong \CH^i(\BZ_f(2))$ pour $i\geq 3$ et $\CH^3(\BZ_X(2))=0$. 
Ainsi  
on a une suite exacte longue dans la cat\'egorie des faisceaux \'etales sur $X$ :
$$\xymatrix@R=1em{0\ar[r]&\CH^0(\BZ_f(2))\ar[r]&\CH^1(\BZ_X(2))\ar[r]^s&R^1f_*\BZ_Y(2)\ar[r]&\\
\CH^1(\BZ_f(2))\ar[r]&\CH^2(\BZ_X(2))\ar[r]^{s_1}&R^2f_*\BZ_Y(2)\ar[r]&\CH^2(\BZ_f(2))\ar[r]&0.
}$$

Les faisceaux $\CH^1(\BZ_X(2))$ et $R^1f_*\BZ_Y(2)$ sont les faisceaux \'etales associ\'es aux pr\'efaisceaux:
$$U\lto \BH^1(U,\BZ(2))=K_{3,ind}k(U)\ \ \mathrm{et}\ \ U\lto \BH^1(f^{-1}U,\BZ(2))=K_{3,ind}k(f^{-1}U).$$
Localement pour la topologie \'etale, $f^{-1}U\iso U\times \BG^N_m$, et dans ce cas on a $K_{3,ind}k(U)\cong K_{3,ind}k(f^{-1}U)$, car $K_{3,ind}k(U)\cong K_{3,ind}k(U\times \BG_m^N)$ (cf. \cite[Lem. 6.2]{EKLV}). Donc $s$ est un isomorphisme.

Les faisceaux $\CH^2(\BZ_X(2))$ et $R^2f_*\BZ_Y(2)$ sont les faisceaux \'etales associ\'es aux pr\'efaisceaux:
$$U\lto \BH^2(U,\BZ(2))=H^0(U,\CK_2)\ \ \mathrm{et}\ \ U\lto \BH^2(f^{-1}U,\BZ(2))=H^0(f^{-1}U,\CK_2).$$
Localement pour la topologie \'etale, $f^{-1}U\iso U\times \BG^N_m$, et dans ce cas le morphisme canonique $H^0(U,\CK_2)\ra H^0(f^{-1}U,\CK_2)$ est injectif 
puisque que $f^{-1}U \to U$ admet alors une section.
Donc $s_1$ est injectif.

On a donc \'etabli $\CH^i(\BZ_f(2))=0$  pour $i\leq 1$.

Le faisceau $R^3f_*\BZ_Y(2)$ est le faisceau \'etale associ\'e au pr\'efaisceau:
$$U\lto \BH^3(f^{-1}U,\BZ(2))=H^1(f^{-1}U,\CK_2).$$
Localement pour la topologie \'etale, $f^{-1}U\iso U\times \BG^N_m$. 
Notons $T_0:=\BG_m^N$.
Dans ce cas, par le lemme \ref{Ktore},
on a un isomorphisme:
$$ H^1(U,\CK_2)\oplus [H^1(U,\CK_1)\otimes T_0^*]\to H^1(f^{-1}U,\CK_2).$$
Puisque $\CK_*$ est un faisceau de Zariski, le faisceau associ\'e \`a \{$U\lto H^1(U,\CK_*)$\} est 0
et donc  le faisceau associ\'e \`a \{$U\lto H^1(f^{-1}U,\CK_2)$\} est 0.
Ainsi $R^3f_*\BZ_Y(2)$ est 0, et $\CH^3(\BZ_f(2))=0$.
\qed
\end{proof}

Notons  $\BZ_f(\otimes )$ le c\^one du morphisme $\BZ_X(1)\otimes^L \BZ_X(1)\to R^{\leq 2}f_*\BZ_Y(1)\otimes^L R^{\leq 2}f_*\BZ_Y(1)$. 
Soit $\theta$ la composition des morphismes
$$R^{\leq 2}f_*\BZ_Y(1)\otimes^L R^{\leq 2}f_*\BZ_Y(1)\to Rf_*\BZ_Y(1)\otimes^L Rf_*\BZ_Y(1) \xrightarrow{\theta_1} $$
$$Rf_*(\BZ_Y(1)\otimes^L f^*Rf_*\BZ_Y(1)) \xrightarrow{\theta_2}Rf_*(\BZ_Y(1)\otimes^L \BZ_Y(1))\to Rf_*\BZ_Y(2).$$
o\`u $\theta_1$ est le morphisme canonique induit par $\otimes^L$ (cf. \cite[p. 306]{Fu})
 et $\theta_2$ est induit par le morphisme d'adjonction $f^*Rf_*\BZ_Y(1)\to \BZ_Y(1)$.
Le morphisme $\theta$ induit un diagramme commutatif de triangles :
\begin{equation}\label{h11def}
\begin{gathered}\xymatrix@C=1.7em{\BZ_X(1)\otimes^L \BZ_X(1)\ar[r]\ar[d]^=&\BZ_X(1)\otimes^L R^{\leq 2}f_*\BZ_Y(1)\ar[r]^{id\times d(1)}\ar[d]^{f^{\#}(1)\times id}&\BZ_X(1)\otimes^L \BZ^{\leq 2}_f(1)\ar[r]\ar@{-->}[d]^{\exists\ h_1}\ar@{}[rd]|{(1)}&\BZ_X(1)\otimes^L \BZ_X(1)[1]\ar[d]\\
\BZ_X(1)\otimes^L \BZ_X(1)\ar[r]\ar[d]^{\cup}&R^{\leq 2}f_*\BZ_Y(1)\otimes^L R^{\leq 2}f_*\BZ_Y(1)\ar[r]\ar[d]^{\theta}&\BZ_f(\otimes )\ar[r]\ar@{-->}[d]^{\exists\ h_{\otimes}}\ar@{}[rd]|{(2)}&\BZ_X(1)\otimes^L \BZ_X(1)[1]\ar[d]^{\cup[1]}\\
\BZ_X(2)\ar[r]&Rf_*\BZ_Y(2)\ar[r]^{d(2)}&\BZ_f(2)\ar[r]&\BZ_X(2)[1], 
}\end{gathered}\end{equation}
o\`u $h_1$ et $h_{\otimes} $ sont donn\'es par les axiomes des cat\'egories triangul\'ees.
Le morphisme $h_{\otimes}\circ h_1$ induit un morphisme
$$h_{1,1}:\ \BG_{m,X}\otimes T^* \stackrel{(\ref{Zf1-e})}{\cong} \CH^2(\BZ_X(1)\otimes^L \BZ^{\leq 2}_f(1))\ra \CH^2(\BZ_f(2)).$$

 Par le Lemme \ref{Ktore}, il y a une composition de morphismes de groupes
 $$H^0(Y\times T,\CK_2)\ra H^0((Y\times T)_{\bk},\CK_2)^{\Gamma_k}\ra (H^0(Y_{\bk},\CK_0)\otimes \wedge^2T^*)^{\Gamma_k}\cong (\wedge^2T^*)^{\Gamma_k}\iso (\wedge^2T^*)(X),$$
 o\`u $(\wedge^2T^*)(X)$ est le groupe de sections du faisceau $(\wedge^2T^*)$ sur $X$.
D\'efinissons:
$$h'_{0,2}(X): H^0(Y,\CK_2)\xrightarrow{\rho^*} H^0(Y\times T,\CK_2)\ra (\wedge^2T^*)(Y)\cong (\wedge^2T^*)(X),$$
o\`u $\rho: Y\times T\to Y: (y,t)\mapsto t\cdot y$ est l'action.
Ceci induit un morphisme de faisceaux:
$$h'_{0,2}: R^2f_*\BZ_Y(2)\to \wedge^2T^*.$$

\begin{lem}\label{Ktoremotivic}
Supposons que $T\cong \BG_m^N$ et $Y\to X$ est un $T$-torseur trivial.
Alors le complexe 
$$0\to \CH^2(\BZ_X(2))\oplus (\BG_{m,X}\otimes T^*)\xrightarrow{ (f^{\#}(2),h_{1,1})} R^2f_*\BZ_Y(2) \xrightarrow{h'_{0,2}}  \wedge^2T^*\to 0$$
est exactement la suite exacte de faisceaux \'etales associ\'e \`a la suite exacte de pr\'efaisceaux qui envoie $U\in \mathrm{Sch}/X$ \`a la suite dans le Lemme \ref{Ktore}.
\end{lem}

\begin{proof}
Dans le diagramme (\ref{h11def}),  la section (\ref{Zf1-e1}) induit une section de $id\times d(1)$: $$s:\  \BZ_X(1)\otimes^L \BZ_f(1) \to \BZ_X(1)\otimes^L R^{\leq 2}f_*\BZ_Y(1) .$$
Donc $h_{1,1}$ est la composition $d(2)\circ\theta\circ (f^{\#}(1)\times id)  \circ s$.
D'apr\`es (\ref{motivice1}), $\theta$ est induit par le cup-produit de groupes de K-th\'eorie  $H^0(U\times T,\CK_1)\otimes H^0(U\times T,\CK_1)\to H^0(U\times T,\CK_2)$  
pour tout $U$ sur $X$,
 et donc $(f^{\#}(2),h_{1,1})$ est le morphisme induit par (\ref{Ktore-e1}) pour tout $ U$ sur $X$.

 Pour tout $U\in \mathrm{Sch}/X$, soit $\phi_U: H^0(U\times T,\CK_2)\to \wedge^2T^*$ le morphisme dans le Lemme \ref{Ktore}. Ceci induit un diagramme avec le carr\'e (5) commutatif
 $$\xymatrixcolsep{5pc}\xymatrix{H^0(X\times T,\CK_2)\ar[r]^-{(id_X\times m)^*}\ar[d]^{\phi_X}\ar@{}[rd]|{(4)}&
 H^0(X\times T\times T,\CK_2)\ar[r]^-{(id_X\times e_T\times id_T)^*}\ar[d]^{\phi_Y\circ (\tau\times id_T)^*}\ar@{}[rd]|{(5)}
 &H^0(X\times T)\ar[d]^{\phi_X}\\
 \wedge^2T^*\ar[r]^=&\wedge^2T^*\ar[r]^=&\wedge^2T^*,
 }$$
 o\`u $m: T\times T\to T$ la multiplication et $\tau: Y\cong X\times T$ une trivialisation.
 Puisque $(id_X\times m) \circ (id_X\times e_T\times id_T)=id_{X\times T}$, le carr\'e (4) est aussi commutatif.
 Donc $h'_{0,2}(X)=\phi_X$, et $h'_{0,2}$ est le morphisme associ\'e \`a $\phi_U$ pour tout $U\in \mathrm{Sch}/X$.
 \qed
\end{proof}

D'apr\`es le Lemme \ref{Ktoremotivic},
la composition $\CH^2(\BZ_X(2))\ra R^2f_*\BZ_Y(2)\xrightarrow{h'_{0,2}} \wedge^2T^*$ est nulle et il existe un morphisme
$$h_{0,2}: \CH^2(\BZ_f(2))\ra \wedge^2T^*.$$

 Nous pouvons maintenant \'etablir :

\begin{prop}\label{cohzf2}
On a $\tau^{\leq 3}\BZ_f(2)\cong \CH^2(\BZ_f(2))[-2]$ et une suite exacte de faisceaux \'etales sur $X$:
\begin{equation}\label{h12}\xymatrix{0\ar[r]&\BG_{m,X}\otimes T^*\ar[r]^{h_{1,1}}&\CH^2(\BZ_f(2))\ar[r]^{h_{0,2}}&\wedge^2T^*\ar[r]&0.
}\end{equation}
\end{prop}

\begin{proof}
D'apr\`es le Lemme \ref{BZf0}, il suffit de montrer que localement pour la topologie \'etale la suite est exacte.
Localement pour la topologie \'etale, $Y\iso X\times \BG_m^N$, et le r\'esultat d\'ecoule donc du Lemme \ref{Ktoremotivic}.
\qed
\end{proof}

Explicitons ce que donne la cohomologie de cette suite exacte de faisceaux.

\begin{thm}\label{mainz2}
Soient $X$ une $k$-vari\'et\'e lisse g\'eom\'etriquement connexe,
 $T$ un $k$-tore et $f:Y\to X$ un $T$-torseur.
Alors on a deux suites exactes longues de $\Gamma_k$-modules:
$$0\to\bk[X]^\times \otimes T^* \to \BH^2(X_{\bk},\BZ_f(2))\to \wedge^2T^*\xrightarrow{h}
\Pic(X_{\bk})\otimes T^*\to \BH^3(X_{\bk},\BZ_f(2))\to 0$$
et
$$\xymatrix@R=1em{0\ar[r]&H^0(X_{\bk},\CK_2)\ar[r]&H^0(Y_{\bk},\CK_2)\ar[r]&\BH^2(X_{\bk},\BZ_f(2))\ar[r]&H^1(X_{\bk},\CK_2)\\
\ar[r]&H^1(Y_{\bk},\CK_2)\ar[r]&\BH^3(X_{\bk},\BZ_f(2))\ar[r]&\BH^4(X_{\bk},\BZ_X(2))\ar[r]&\BH^4(Y_{\bk},\BZ_Y(2)),
}$$
o\`u :
\begin{itemize}
\item[\rm (1)] le morphisme $h$ est donn\'e par: $a\wedge b\mapsto a\otimes \partial (b)-b\otimes \partial (a)$ avec $\partial : T^*\ra \Pic(X_{\bk})$
 le morphisme associ\'e au torseur $Y \to X$ sous le tore $T$ (induit par (\ref{e97}));
\item[\rm (2)] la composition de morphismes
$$\Pic(X_{\bk})\otimes T^*\to \BH^3(X_{\bk},\BZ_f(2))\to \BH^4(X_{\bk},\BZ_X(2)),$$
est induite par l'intersection:
 $$\Pic(X_{\bk})\otimes T^*\xrightarrow{id\times \partial} \Pic(X_{\bk})\times \Pic(X_{\bk})\ra \CCH^2(X_{\bk})\to \BH^4(X_{\bk},\BZ_X(2)).$$
\end{itemize}
\end{thm}

\begin{proof}
Sur $\bk$, on a $T^*\cong \BZ^N$, donc 
$$H^1(X_{\bk},\wedge^2T^*)=0\ \ \ \text{ et}\ \ \  H^i(X_{\bk},\BG_{m,X}\otimes T^*)\cong H^i(X_{\bk},\BG_{m,X})\otimes T^*$$
(le\,premier\,\'enonc\'e\,utilisant\,la\,lissit\'e\,de\,$X$).
D'apr\`es la Proposition\,\ref{cohzf2},\,$\BH^i(X_{\bk},\BZ_f(2))\cong H^{i-2}(X_{\bk},\CH^2(\BZ_f(2)))$ pour chaque $i\leq 3$.
En appliquant $H^i(X_{\bk},-)$ aux suites exactes (\ref{Zde2}) et (\ref{h12}), on d\'eduit 
les deux suites exactes longues.

La composition dans l'\'enonc\'e (2) est induite par $\BZ_X(1)\otimes^L \BZ^{\leq 2}_f(1)\ra \BZ_f(2)\ra \BZ_X(2)[1]$.
D'apr\`es les carr\'es (1), (2) dans le diagramme (\ref{h11def}) et (\ref{motivice2}), on a un diagramme commutatif:
$$\xymatrix{\Pic(X_{\bk})\otimes T^*\ar[r]^-{id\times \partial}\ar[d]&H^2(X_{\bk},\BZ_X(1))\otimes H^2(X_{\bk},\BZ_X(1))\ar[d]\ar@{}[rd]|{(3)}
&\Pic(X_{\bk})\otimes \Pic(X_{\bk})\ar[d]^{\cup}\ar[l]_-{\cong}\\
\BH^3(X_{\bk},\BZ_f(2))\ar[r]&H^4(X_{\bk},\BZ_X(2))&\CCH^2(X_{\bk}),\ar@{_{(}->}[l]^{\eqref{e1}}
}$$
o\`u $\cup$ est l'intersection.
 L'\'enonc\'e (2) en d\'ecoule.

\'Etablissons l'\'enonc\'e (1). 
On a un diagramme commutatif de triangles:
$$\xymatrix{\BZ_X(1)\otimes^L \BZ_X(1)\ar[r]\ar[d]&\BZ_X(1)\otimes^L R^{\leq 2}f_*\BZ_Y(1)\ar[r]\ar[d]&
\BZ_X(1)\otimes^L \BZ^{\leq 2}_f(1)\ar[r]^-{+1}\ar[d]&\\
R^{\leq 2}f_*\BZ_Y(1)\otimes^L \BZ^{\leq 2}_X(1)\ar[r]\ar[d]&R^{\leq 2}f_*\BZ_Y(1)\otimes^L R^{\leq 2}f_*\BZ_Y(1)\ar[r]\ar[d]&
R^{\leq 2}f_*\BZ_Y(1)\otimes^L \BZ^{\leq 2}_f(1)\ar[r]^-{+1}\ar[d]&\\
\BZ^{\leq 2}_f(1)\otimes^L \BZ_X(1)\ar[r]\ar[d]^{+1}&\BZ^{\leq 2}_f(1)\otimes^L R^{\leq 2}f_*\BZ_Y(1)\ar[r]\ar[d]^{+1}&
\BZ^{\leq 2}_f(1)\otimes^L \BZ^{\leq 2}_f(1)\ar[r]^-{+1}\ar[d]^{+1}&\\
&&&.
}$$
 Par la d\'efinition de $\BZ_f(\otimes )$ et la d\'emonstration de \cite[Prop. 1.1.11]{BBD} appliqu\'e au diagramme ci-dessus, 
 on a les triangles (\cite[pp. 25, Diagramme (1) et (2)]{BBD}):
$$\BZ_X(1)\otimes^L \BZ^{\leq 2}_f(1)\xrightarrow{h_1}\BZ_f(\otimes )\to \BZ^{\leq 2}_f(1)\otimes^L R^{\leq 2}f_*\BZ_Y(1)\xrightarrow{+1},$$
$$\BZ^{\leq 2}_f(1)\otimes^L \BZ_X(1)\xrightarrow{h_2}\BZ_f(\otimes )\to R^{\leq 2}f_*\BZ_Y(1)\otimes^L \BZ^{\leq 2}_f(1) \xrightarrow{+1}$$
et donc on a un diagramme commutatif de triangles (le lemme \ref{leme97} ci-dessous):
\begin{equation}\label{e98} \begin{gathered} \xymatrixcolsep{1pc} \xymatrix{
[\BZ_X(1)\otimes^L \BZ^{\leq 2}_f(1)]\oplus [\BZ^{\leq 2}_f(1)\otimes^L \BZ_X(1)]\ar[r]^-{h_1+h_2}\ar[d]^{pr_1}&\BZ_f(\otimes)\ar[r]\ar[d]& \BZ^{\leq 2}_f(1)\otimes^L \BZ^{\leq 2}_f(1)\ar[r]^-{+1}\ar[d]&\\
\BZ_X(1)\otimes^L \BZ^{\leq 2}_f(1)\ar[r]& R^{\leq 2}f_*\BZ_Y(1)\otimes^L \BZ^{\leq 2}_f(1)\ar[r] & \BZ^{\leq 2}_f(1)\otimes^L \BZ^{\leq 2}_f(1)\ar[r]^-{+1}&
.} \end{gathered}\end{equation}

Notons   $\BZ(\wedge)$ le c\^one de $\BZ_X(1)\otimes^L \BZ^{\leq 2}_f(1)\xrightarrow{h_{\otimes}\circ h_1}\BZ_f(2)$, 
o\`u $\BZ_f(\otimes)\xrightarrow{h_{\otimes}}\BZ_f(2)$ est le morphisme du diagramme (\ref{h11def}). 
Par les Propositions \ref{Zf1} et \ref{cohzf2}, on a $\tau^{\leq 3}\BZ(\wedge)\cong \wedge^2T^*[-2]$.

Notons    $\tau: \BZ^{\leq 2}_f(1)\otimes^L \BZ_X(1)\to \BZ_X(1)\otimes^L \BZ^{\leq 2}_f(1)$ l'isomorphisme canonique.
 Puisque le produit $\BZ(1)\otimes^L \BZ(1)\to \BZ(2)$ est sym\'etrique, on a $h_{\otimes }\circ h_1\circ \tau =h_{\otimes } \circ h_2$ et un diagramme commutatif de triangles :
\begin{equation}\label{e99}
\begin{gathered}
\xymatrix{[\BZ_X(1)\otimes^L \BZ^{\leq 2}_f(1)]\oplus [\BZ^{\leq 2}_f(1)\otimes^L \BZ_X(1)]\ar[r]^-{h_1+h_2}\ar[d]^{id+ \tau}& \BZ_f(\otimes)\ar[r]\ar[d]^{h_{\otimes}}&
 \BZ^{\leq 2}_f(1)\otimes^L \BZ^{\leq 2}_f(1)\ar[r]^-{+1}\ar@{-->}[d]^{h_{\cup}}&\\
\BZ_X(1)\otimes^L \BZ^{\leq 2}_f(1)\ar[r]^-{h_{\otimes}\circ h_1}&\BZ_f(2)\ar[r]&\BZ(\wedge)\ar[r]^{+1}&,
}\end{gathered}\end{equation}
et l'homomorphisme $h$ est induit par le deuxi\`eme triangle.
Puisque $[\BZ^{\leq 2}_f(1)\otimes^L \BZ^{\leq 2}_f(1)]\cong (T^*\otimes T^*)[-2]$, 
le morphisme $h_{\cup}$ induit un morphisme de faisceaux $\bk$-constants: $T^*\otimes T^* \xrightarrow{h'_{\cup}} \wedge^2T^*$.
On peut calculer $h'_{\cup}$ en se ramenant au cas  $Y\iso X\times T$, et dans ce cas, $h'_{\cup}$ est exactement  le cup-produit $a\otimes b\to a\wedge b$.
En g\'en\'eral, $h'_{\cup}$ est donc le cup-produit.
Appliquons $\CH^2(-)$ au diagrammes (\ref{e98}) et (\ref{e99}), on a un diagramme commutatif \`a lignes exactes
$$\xymatrix{0\ar[r]&\BG_{m,X}\otimes T^*\ar[r]&R^1f_*\BZ_Y(1) \otimes T^*   \ar[r]& T^*\otimes T^*\ar[r]&0\\
0\ar[r]& (\BG_{m,X}\otimes T^*)\oplus (T^*\otimes \BG_{m,X}) \ar[u]^{pr_1} \ar[r]\ar[d]^{id+ \tau_1}&\CH^2(\BZ_f(\otimes))\ar[r]\ar[d]\ar[u]&T^*\otimes T^*\ar[d]^{\mathrm{cup-produit}}\ar[r]\ar[u]^=&0\\
0\ar[r]&\BG_{m,X}\otimes T^*\ar[r]&\CH^2(\BZ_f(2))\ar[r]&  \wedge^2T^*\ar[r]&0
}$$
o\`u la premi\`ere ligne est (\ref{e97})$\otimes T^*$, la troisi\`eme ligne est la suite exacte (\ref{h12}) et $\tau_1$ l'homomorphisme induit par $\tau$ en degr\'e $1$, 
i.e. $$\tau_1: \ T^*\otimes \BG_{m,X}\to  \BG_{m,X}\otimes T^*: t\otimes x \mapsto -x\otimes t.$$

Appliquons $H^i(X_{\bk},-)$ au diagramme ci-dessus, on a un diagramme commutatif:
$$\xymatrixcolsep{5pc}\xymatrix{T^*\otimes T^*\ar[d]^{\partial \times id_{T^*}}&T^*\otimes T^*\ar[r]^{\mathrm{cup-produit}}\ar[l]_=\ar[d]^{\partial_{\otimes}}&\wedge^2 T^*\ar[d]^h\\
\Pic(X_{\bk})\otimes T^*& \Pic(X_{\bk})\otimes T^* \oplus T^*\otimes \Pic(X_{\bk})\ar[r]^-{id+\tau_1}\ar[l]_-{pr_1} &\Pic(X_{\bk})\otimes T^*.
}$$
 Donc $pr_1(\partial_{\otimes}(a\otimes b))=\partial(a)\otimes b$. Par le m\^eme argument, $pr_1(\partial_{\otimes}(a\otimes b))=a\otimes \partial(b)$. 
 Ainsi on a:
 $\partial_{\otimes}(a\otimes b)=(\partial(a)\otimes b,a\otimes \partial(b))$ et $h(a\wedge b)=a\otimes \partial (b)-b\otimes \partial (a)$.
 \qed
\end{proof}

\begin{lem}\label{leme97}
Soient $\CD$ une cat\'egorie triangul\'ee et $A_i\xrightarrow{h_i} B\xrightarrow{g_i} C_i\xrightarrow{+1}$ deux triangles pour $i=1,2$. Alors on a un diagramme commutatif de triangles:
$$\xymatrixcolsep{5pc}\xymatrix{A_1\oplus A_2\ar[r]^{h_1+h_2}\ar[d]_{pr_2}&B\ar[r]\ar[d]^{g_1}&D\ar[r]^{+1}\ar[d]^=&\\
A_2\ar[r]^{g_1\circ h_2}&C_1\ar[r]&D\ar[r]^{+1}&
}$$ 
o\`u $D$ est le c\^one de $(g_1\circ h_2)$.
\end{lem}

\begin{proof}
Puisque $\mathrm{Cone}(pr_2)\cong A_1[1]\cong \mathrm{Cone}(g_1)$,
d'apr\`es \cite[Exer. 10.6]{KS}, on a un diagramme commutatif de triangles:
$$\xymatrixcolsep{5pc}\xymatrix{A_1\ar[r]^=\ar[d]&A_1\ar[r]\ar[d]&0\ar[r]^{+1}\ar[d]&\\
A_1\oplus A_2\ar[r]^{h_1+h_2}\ar[d]_{pr_2}&B\ar[r]\ar[d]^{g_1}&\mathrm{Cone}(h_1+h_2)\ar[r]^{+1}\ar[d]&\\
A_2\ar[r]^{g_1\circ h_2}\ar[d]^{+1}&C_1\ar[r]\ar[d]^{+1}&D\ar[r]^{+1}\ar[d]^{+1}&\\
&&&.
}$$
Donc $\mathrm{Cone}(h_1+h_2)\cong \mathrm{Cone}(g_1\circ h_2)\cong D$, d'o\`u le r\'esultat.
\qed
\end{proof}

\section{Cohomologie motivique \`a coefficients $\BZ(2)$ des torseurs universels sur une surface g\'eom\'etriquement rationnelle}

Soit $k$ un corps de caract\'eristique 0.
 Soient $X$ une surface projective lisse g\'eom\'etriquement rationnelle sur $k$, et $g: \CT\ra X$ un torseur universel de $X$. 
 Le r\'esultat principal de cette section est le Th\'eor\`eme \ref{H4Z2}.

\begin{lem}\label{codim1}
Supposons que $k$ est alg\'ebriquement clos.
Soit $X$ une $k$-vari\'et\'e projective, lisse, rationnelle, connexe. 
Soit $\CT \ra X$ un  torseur universel sous le tore $T$. 
Soit $T^c$ une $T$-vari\'et\'e torique, projective, lisse.
Soit  $\CT^{c}= \CT\times^TT^c$.
Soit $Z$ une $T$-orbite de codimension $l$ dans $T^c$. Notons $\CZ:=Z\times^{T}\CT\sbt \CT^c$. Alors $\CZ$ est lisse, $k[\CZ]^{\times}=k^{\times}$, $\Pic(\CZ)\cong \BZ^l$, $\Br(\CZ)=0$, $H^1(\CZ,\BZ/n)=0$, $H^2_{\nr}(\CZ,\BQ/\BZ(1))=0$ et $H^2(\CZ, \BZ/n(1))\cong (\BZ/n)^l$.
\end{lem}

\begin{proof}
Notons $l':=\dim (T)-l$.
Puisque $Z$ est une $T$-orbite de codimension $l$, par  le Th\'eor\`eme \ref{torique}, 
il y a une sous-vari\'et\'e torique affine $U\sbt T^c$ telle que $U\cong \BA^{l}\times \BG_m^{l'}$ comme vari\'et\'e torique et $Z\sbt U$.
Alors il y a un homomorphisme $\BG^l_m\ra T$
 tel que $Z\iso T/\BG^l_m$. 
 Notons $\phi: T\to  T/\BG^l_m \cong Z$.
 Alors $\CZ:=(T/\BG^l_m)\times^{T}\CT$.
 Par la construction g\'eom\'etrique de $H^1(X, T)\xrightarrow{\phi_*} H^1(X, T/\BG^l_m)$, 
 le morphisme $\CZ\to X$ est un torseur sous $T/\BG^l_m$ et  $[\CZ]=\phi_*([\CT])$.
  Ainsi $\CZ$ est lisse. Par \cite[\S 2.1]{CTS} (ou 
 \cite[Prop. 6.10]{S}),
 on a un diagramme commutatif \`a lignes exactes :
$$\xymatrix{0\ar[r]&k[X]^{\times}/k^{\times}\ar[r]\ar[d]^=&k[\CZ]^{\times}/k^{\times}\ar[r]\ar[d]&(T/\BG^l_m)^*\ar@{^{(}->}[d]\ar[r]&\Pic(X)\ar[r]\ar[d]^=&
\Pic(\CZ)\ar[r]\ar[d]&0\\
0\ar[r]&k[X]^{\times}/k^{\times}\ar[r]&k[\CT]^{\times}/k^{\times}\ar[r]&T^*\ar[r]^{\cong}&\Pic(X)\ar[r]&\Pic(\CT)\ar[r]&0.
}$$
Puisque $\Pic(\CT)=0$ (Th\'eor\`eme \ref{torsoruniv}), on a $k[\CZ]^{\times}=k^{\times}$ et $\Pic(\CZ)\cong \BZ^l$. Par \cite[Th\'eor\`eme 1.6]{HS} 
et le fait que $X$ est rationnelle, on a $\Br(\CZ)=0$. 
Ainsi $H^2_{\nr}(\CZ,\BQ/\BZ(1))=0$.
Le reste du r\'esultat se d\'eduit de la suite de Kummer.
\qed
\end{proof}

\begin{lem}\label{lem99}
Supposons que $k$ est alg\'ebriquement clos.
Soient $X$ une $k$-vari\'et\'e projective, lisse, rationnelle et $\CT \ra X$ un  torseur universel.
 Le morphisme naturel $H^1(X,\CK_2)\to H^1(\CT ,\CK_2)$ est nul.
\end{lem}

\begin{proof}
On a un diagramme commutatif:
$$\xymatrix{\Pic (X)\otimes k^{\times}\ar[r]^{\theta}\ar[d]&H^1(X,\CK_2)\ar[d]\\
0=\Pic (\CT )\otimes k^{\times}\ar[r]& H^1(\CT, \CK_2).
}$$
D'apr\`es \cite[Prop. 2.6]{Pi}, le morphisme $\theta$ est un isomorphisme.
On a $\Pic (\CT )=0$ par le Th\'eor\`eme \ref{torsoruniv}.
Donc $H^1(X,\CK_2)\to H^1(\CT ,\CK_2)$ est nul. 
\qed
\end{proof}

\begin{thm}\label{S2T}
Soit $X$ une $k$-vari\'et\'e projective, lisse, g\'eom\'e\-triquement int\`egre et g\'eom\'e\-triquement rationnelle.
Soit $g: \CT\to X$ un torseur universel de $X$. 
Alors on a $H^3_{\nr}(\CT_{\bk},\BQ/\BZ(2))=0$ et une suite exacte naturelle de $\Gamma_k$-modules:
\begin{equation}\label{S2T-e}
 0 \to H^1(\CT_{\bk},\CK_2)\to \Sym^2\Pic(X_{\bk})\xrightarrow{\cup}\CCH^2(X_{\bk})\to \CCH^2(\CT_{\bk})\to 0,
 \end{equation}
o\`u  $\cup$ est induit par le produit d'intersection 
$\Pic(X_{\bk})\times \Pic(X_{\bk})\to \CCH^2(X_{\bk}).$
\end{thm}

\begin{proof}
Soient $T^c$ une $T$-vari\'et\'e torique, projective, lisse et $\CT^{c}:= \CT\times^TT^c$.
Puisque $X_{\bk}$ est rationnelle, $H^3_{\nr}(X_{\bk},\BQ/\BZ(2))=0$  et $\CT_{\bk}$ est rationnelle.
Ainsi $H^3_{\nr}(\CT^c_{\bk},\BQ/\BZ)=0$. 
 
Par la r\'esolution de  Gersten, on a une suite exacte:
$$0\to H^3_{\nr}(\CT^c_{\bk},\BQ/\BZ(2))\to H^3_{\nr}(\CT_{\bk},\BQ/\BZ(2))\to \ker(\psi)$$
o\`u 
$$\psi: \oplus_{x\in (\CT^c_{\bk}\setminus \CT_{\bk})^{(0)}}H^2(\bk(x),\BQ/\BZ(1))\to \oplus_{x\in (\CT^c_{\bk}\setminus \CT_{\bk})^{(1)}}H^1(\bk(x),\BQ/\BZ).$$
Soit $S_1$ l'ensemble des $T_{\bk}$-orbites de codimension $1$ de $T^c_{\bk}$.
Pour $Z\in S_1$, notons $\CZ:=Z\times^{T}\CT\sbt \CT^c$.
Alors $S_1\cong (\CT^c_{\bk}\setminus \CT_{\bk})^{(0)}$ et
$$\ker(\psi)\sbt \oplus_{Z\in S_1} H^2_{\nr}(\CZ,\BQ/\BZ(1))\sbt \oplus_{x\in (\CT^c_{\bk}\setminus \CT_{\bk})^{(0)}}H^2(\bk(x),\BQ/\BZ(1)).$$ 
D'apr\`es le Lemme \ref{codim1}, $\ker(\psi)=0$, et donc $H^3_{\nr}(\CT_{\bk},\BQ/\BZ(2))=0$.

 D'apr\`es le Th\'eor\`eme \ref{mainz2}, 
 $$\BH^3(X_{\bk},\BZ_g(2))\cong T^*\otimes T^*/\wedge^2T^* \cong \Sym^2T^* \cong \Sym^2\Pic(X_{\bk}).$$

 Par le Th\'eor\`eme \ref{motivic}, $\CCH^2(X_{\bk})\cong \BH^4(X_{\bk},\BZ(2))$ et $\CCH^2(\CT_{\bk})\cong \BH^4(\CT_{\bk},\BZ(2))$.
Pour tout $x\in X_{\bk}$,  la fibre $\CT_{\bk,x}\cong \BG^{dim(T)}_{m,k(x)}$ et donc $\CCH^i(\CT_{\bk,x})=0$ pour tout $i\geq 1$.
D'apr\`es \cite[Cor. 8.2]{Ro}, le morphisme de restriction $\CCH^2(X_{\bk})\to \CCH^2(\CT_{\bk})$ est surjectif.
 On conclut alors avec le Lemme \ref{lem99} et le Th\'eor\`eme \ref{mainz2}.
 \qed
\end{proof}

\begin{cor}\label{S2Tcor}
Sous les hypoth\`eses du Th\'eor\`eme \ref{S2T}, soient $U\sbt X$ un ouvert et $\CT_U:=\CT\times_XU$.
Supposons que $\codim(X\setminus U,X)\geq 2$ et $H^1(X_{\bk},\CK_2)\cong H^1(U_{\bk},\CK_2)$.
Alors on a $H^3_{\nr}(\CT_{U, \bk},\BQ/\BZ(2))=0$ et une suite exacte naturelle de $\Gamma_k$-modules:
$$ 0 \to H^1(\CT_{U,\bk},\CK_2)\to \Sym^2\Pic(U_{\bk})\to \CCH^2(U_{\bk})\to \CCH^2(\CT_{U,\bk})\to 0.$$
\end{cor}

\begin{proof}
Puisque la cohomologie non ramifi\'ee ne d\'epend pas des sous-ensembles de codimension $\geq 2$, 
on a $H^3_{\nr}(\CT_{U, \bk},\BQ/\BZ(2))=0$ et $H^3_{\nr}(U_{\bk},\BQ/\BZ(2))=0$. 
Par le Th\'eor\`eme \ref{motivic}, $\CCH^2(U_{\bk})\cong \BH^4(U_{\bk},\BZ(2))$ 
et $\CCH^2(\CT_{U,\bk})\cong \BH^4(\CT_{U, \bk},\BZ(2))$.
D'apr\`es le Th\'eor\`eme \ref{mainz2}, on a un diagramme commutatif de suites exactes:
$$\xymatrix{H^1(X_{\bk},\CK_2)\ar[r]^0\ar[d]^{\cong}&H^1(\CT_{\bk},\CK_2)\ar[r]\ar[d]&\Sym^2\Pic(X_{\bk})\ar[r]\ar[d]^{\cong}&
\CCH^2(X_{\bk})\ar[r]\ar@{->>}[d]&\CCH^2(\CT_{\bk})\ar[r]\ar@{->>}[d]&0\\
H^1(U_{\bk},\CK_2)\ar[r]&H^1(\CT_{U,\bk},\CK_2)\ar[r]&\Sym^2\Pic(U_{\bk})\ar[r]&\CCH^2(U_{\bk})\ar[r]&\CCH^2(\CT_{U,\bk})&.
}$$
Le r\'esultat en d\'ecoule.
\qed
\end{proof}

\begin{prop}\label{coh}
Supposons que $k$ est alg\'ebriquement clos. Soient $X$ une $k$-surface projective lisse rationnelle et $\CT$ un  torseur universel sur $X$. 
Alors $\CCH^2(\CT)=0$.
\end{prop}

\begin{proof}
On a un diagramme commutatif:
$$\xymatrix{\Pic(X)\times \Pic(X)\ar@{->>}[d]\ar[r]&\Pic(\CT)\times \Pic(\CT)\ar[d]\\
\CCH^2(X)\ar[r]&\CCH^2(\CT).
}$$
D'apr\`es le Th\'eor\`eme \ref{S2T}, le morphisme $\CCH^2(X)\to \CCH^2(\CT)$ est surjectif.
Donc le morphisme $\Pic(\CT)\times \Pic(\CT)\ra \CCH^2(\CT)$ est surjectif. Puisque $\Pic(\CT)=0$, on a $\CCH^2(\CT)=0$.
\qed
\end{proof}

Une cons\'equence imm\'ediate du Th\'eor\`eme \ref{S2T} et de la Proposition \ref{coh} est:

\begin{cor}\label{corsurface}
 Soient $X$ une $k$-surface projective lisse g\'eom\'e\-triquement rationnelle et $\CT$ un torseur universel sur $X$. On a alors une suite exacte naturelle de r\'eseaux galoisiens
$$ 0 \to H^1(\CT_{\bk},\CK_2)\to \Sym^2\Pic(X_{\bk})\xrightarrow{\cup}  \BZ \to 0 .$$
\end{cor}

\begin{thm}\label{H4Z2}
Soient $X$ une $k$-surface projective lisse g\'eom\'etriquement rationnelle et $\CT$ un torseur universel sur $X$. Alors
\begin{itemize}
\item[\rm (1)] Les groupes $\BH^4(\CT,\BZ (2))/\BH^4(k,\BZ (2))$, $\frac{H^3_{\nr}(\CT,\BQ/\BZ(2))}{H^3(k,\BQ/\BZ(2))}$ et $\CCH^2(\CT)$ sont finis;
\item[\rm (2)] On a des suites exactes $$0\to \BH^4(\CT,\BZ (2))/\BH^4(k,\BZ (2))\to H^1(k,H^1(\CT_{\bk},\CK_2))\to \Ker(\BH^5(k,\BZ(2))\to \BH^5(\CT,\BZ(2)))$$
et
$$(\Sym^2\Pic(X_{\bk}))^{\Gamma_k}\xrightarrow{\cup}\BZ\cong \CCH^2(X_{\bk})\ra H^1(k,H^1(\CT_{\bk},\CK_2))\ra H^1(k,\Sym^2\Pic(X_{\bk}))\ra 0;$$
\item[\rm (3)] Si $H^1(k,\Sym^2\Pic(X_{\bk}))=0$ et $X(k)\neq \emptyset$, alors $\frac{H_{\nr}^3(\CT,\BQ/\BZ(2))}{H^3(k,\BQ/\BZ(2))}=0$.
\end{itemize}
\end{thm}

\begin{proof}
Par la Proposition \ref{coh} et le Th\'eor\`eme \ref{S2T}, $\CCH^2(\CT_{\bk})=0$, $\BH^4(\CT_{\bk},\BZ (2))=0$ et $H^1(\CT_{\bk},\CK_2)$ est de type fini sans torsion.
Ainsi $H^1(k,H^1(\CT_{\bk},\CK_2))$ est fini.
Par \cite[\S 2.2]{CT1} et le fait que $H^0(\CT_{\bk},\CK_2)$ est uniquement divisible (Corollaire \ref{H0K2UD}), on a une suite exacte:
\begin{equation}\label{thmH4Z2e}
\BH^4(k,\BZ (2))\ra \BH^4(\CT,\BZ (2))\ra H^1(k,H^1(\CT_{\bk},\CK_2))\to \Ker(\BH^5(k,\BZ(2))\to \BH^5(\CT,\BZ(2))).
\end{equation}
La premi\`ere suite exacte de (2) en d\'ecoule. 
Par la suite exacte (\ref{e1}), les groupes $\frac{H^3_{\nr}(\CT,\BQ/\BZ(2))}{H^3(k,\BQ/\BZ(2))}$ et $\CCH^2(\CT)$ sont finis.

Puisque $\BH^4(X_{\bk},\BZ(2))\cong \CCH^2(X_{\bk})\cong \BZ$, on a $H^1(k,\BH^4(X_{\bk},\BZ(2)))=0$.   
Une application  du Corollaire \ref{corsurface} donne la deuxi\`ere suite exacte de (2).

Soit $F\sbt \CT$ un sous-ensemble ferm\'e de codimension $\geq 2$.
Alors $\CCH^2((\CT\setminus F)_{\bk})=0$,
$$H^3_{\nr}((\CT\setminus F)_{\bk},\BQ/\BZ(2))\cong H^3_{\nr}(\CT_{\bk},\BQ/\BZ(2))=0\ \ \ \text{et donc}\ \ \ H^4((\CT\setminus F)_{\bk},\BZ(2))=0.$$
Par la r\'esolution de Gersten (\ref{resolutiongersten}), on a un isomorphisme 
entre $H^0(\CT_{\bk},\CK_2)$ et $H^0((\CT\setminus F)_{\bk},\CK_2)$
et donc $H^0((\CT\setminus F)_{\bk},\CK_2)$ est uniquement divisible.
  D'apr\`es \cite[\S 2.2]{CT1}, la suite exacte (\ref{thmH4Z2e}) vaut aussi pour $\CT\setminus F$. Donc on a  un morphisme naturel injectif: 
\begin{equation}\label{H4Z2-e}
\BH^4(\CT\setminus F,\BZ(2))/\BH^4(k,\BZ(2))\hookrightarrow H^1(k,H^1((\CT\setminus F)_{\bk},\CK_2)),
\end{equation}
qui est un isomorphisme si $(\CT\setminus F)(k)\neq\emptyset $.

Sous les hypoth\`eses de (3), soient $x\in X(k)$, $U:=X\setminus \{x\}$ et $\CT_U:=\CT\times_XU$.
Alors $\codim(\CT\setminus \CT_U,\CT)\geq 2$.
Par la r\'esolution de Gersten (\ref{resolutiongersten}), on a une suite exacte:
$$0\to H^1(X_{\bk},\CK_2)\to H^1(U_{\bk},\CK_2)\to \BZ\cdot x\xrightarrow{\phi} \CCH^2(X_{\bk})\to \CCH^2(U_{\bk})\to 0.$$
Puisque $\phi$ est un isomorphisme, on a $H^1(X_{\bk},\CK_2)\cong H^1(U_{\bk},\CK_2)$ et $\CCH^2(U_{\bk})=0$.
D'apr\`es le Corollaire \ref{S2Tcor}, $H^1(\CT_{U,\bk},\CK_2)\cong \Sym^2\Pic(X_{\bk})$.
Par hypoth\`eses, $H^1(k, H^1(\CT_{U,\bk},\CK_2))=0$.
D'apr\`es (\ref{H4Z2-e}), le groupe $\BH^4(\CT_U,\BZ(2))/$ $\BH^4(k,\BZ(2))$ est trivial.
Une  application du fait $H^3_{\nr}(\CT,\BQ/\BZ(2))\cong H^3_{\nr}(\CT_U,\BQ/\BZ(2))$ donne (3).
\qed
\end{proof}

\begin{rem}{\rm
Dans le Th\'eor\`eme \ref{H4Z2}, le crit\`ere $H^1(k,\Sym^2\Pic(X_{\bk}))=0$ n'est pas birationnellement invariant.
En fait, soit $X$ une surface projective lisse g\'eom\'e\-triquement rationnelle avec $X(k)\neq\emptyset$ 
et $X'$  un \'eclatement de $X$ en  un $k$-point.
Supposons   $H^1(k,\Pic(X_{\bk}))\neq 0$ (cf. \cite[\S 31]{Ma} pour des exemples).
Alors $\Pic(X'_{\bk})\iso \Pic(X_{\bk})\oplus \BZ$ et
 $$H^1(k,\Sym^2\Pic(X'_{\bk}))\cong H^1(k,\Sym^2\Pic(X_{\bk})\oplus \Pic(X_{\bk})\oplus \BZ).$$
Alors, m\^eme si  $H^1(k,\Sym^2\Pic(X_{\bk}))=0$,
on a $H^1(k,\Sym^2\Pic(X'_{\bk}))\neq 0$.}
\end{rem}

\begin{cor}\label{fini}
\ \!\!\!Sous\,les\,hypoth\`eses\,du\,Th\'eor\`eme\,\ref{H4Z2},\,le\,morphisme\,canonique\,$H^3(\CT,\BQ/\BZ(2)) \to H^3_{\nr}(\CT,\BQ/\BZ(2))$ est surjectif.
\end{cor}

\begin{proof}
Par \cite[Suite 3.10]{CT2}, on a une suite exacte:
$$0\ra NH^3(\CT,\BQ/\BZ(2))\ra H^3(\CT,\BQ/\BZ(2))\ra H^3_{\nr}(\CT,\BQ/\BZ(2))\ra \CCH^2(\CT)\otimes \BQ/\BZ.$$
Par le Th\'eor\`eme \ref{H4Z2}, $\CCH^2(\CT)$ est fini, et donc $\CCH^2(\CT)\otimes \BQ/\BZ=0$. Donc le morphisme $H^3(\CT,\BQ/\BZ(2))\ra H^3_{\nr}(\CT,\BQ/\BZ(2))$ est surjectif.
\qed
\end{proof}

\section{Surfaces de Ch\^atelet g\'en\'eralis\'ees}

Dans cette section,  $X$ est une surface de Ch\^atelet g\'en\'eralis\'ee, 
c'est-\`a-dire un fibr\'e en coniques sur $\BP^1_{k}$ poss\'edant une section sur une extension quadratique $k'$ de $k$.
Soient $\CT$ un torseur universel de $X$ et $\CT^c$ une compactification lisse de $\CT$.
Alors $X_{k'}$ est $k'$-rationnel et, par le corollaire \ref{zerorationnel}, 
$\frac{H_{\nr}^3(\CT^c,\BQ/\BZ(2))}{H^3(k,\BQ/\BZ(2))}$ est annul\'e par~$2$.

Le r\'esultat principal de ce paragraphe est:

\begin{thm}\label{coniccoh}
Soient $X$ une vari\'et\'e lisse projective, $P(t)$ est un polyn\^ome s\'eparable et supposons que $X$ contienne un ouvert de la forme $\Spec\ k[y,z,t]/(y^2-az^2=P(t))$.
Soient $\CT$ un torseur universel de $X$ et $\CT^c$ une compactification lisse de $\CT$. Supposons que $X(k)\neq \emptyset$.
Alors $\frac{H_{\nr}^3(\CT^c,\BQ/\BZ(2))}{H^3(k,\BQ/\BZ(2))}=0$.
\end{thm}

Dans ce th\'eor\`eme, $P(t)$ est un polyn\^ome s\'eparable de degr\'e quelconque 
(dans la d\'efinition des surfaces de Ch\^atelet g\'en\'eralis\'ees dans \cite{CT87}, $P(t)$ est de degr\'e 3 ou 4).

\begin{proof}
Si $X$ et $X_1$ sont birationnellement \'equivalents, 
alors $X_1(k)\neq\emptyset$ et il existe un torseur universel $\CT_1^c$ de $X_1$ tel que $\CT^c$ et $\CT_1^c$ soient stablement birationnellement \'equivalents par \cite[Prop. 2.9.2]{CTS}.
Donc on a $H_{\nr}^3(\CT^c,\BQ/\BZ(2))\iso H_{\nr}^3(\CT_1^c,\BQ/\BZ(2))$. C'est-\`a-dire que l'on peut remplacer $X$ par $X_1$.
Ainsi, comme  dans \cite[p. 135]{sko01}, on peut supposer que $X$ est une vari\'et\'e qui v\'erifie les conditions de  la Proposition \ref{conicsym} ci-dessous.
Par la Proposition \ref{conicsym}, on a $H^1(k,\Sym^2\Pic(X_{\bk}))=0$. D'apr\`es le Th\'eor\`eme \ref{H4Z2}, on a
$\frac{H_{\nr}^3(\CT^c,\BQ/\BZ(2))}{H^3(k,\BQ/\BZ(2))}=0$.
\qed
\end{proof}

\begin{prop}\label{conicsym}
Soient $P(t)\in k[t]$ un polyn\^ome s\'eparable de degr\'e pair, $(P_i(t))_{i\in I}$ ses facteurs irr\'eductibles, et $a\in k$ un \'el\'ement tel que $a$ n'est un carr\'e de $k[t]/P_i(t)$ pour aucun $i\in I$ (i.e. chaque $P_i(t)$ est irr\'eductible dans $k(\sqrt{a})[t]$). Soit $X_1$ (resp. $X_2$) une sous-vari\'et\'e ferm\'ee de $\RP^2\times (\RP^1\setminus\infty)$
(resp. de $\RP^2\times (\RP^1\setminus 0)$) avec les coordonn\'ees $(y:z:u;t)$ (resp. $(y':z':u';t')$) donn\'ee par $y^2-az^2-P(t)u^2=0$
(resp. $(y')^{2}-a(z')^{2}-P((t')^{-1})(u')^{2}=0$). On recolle $X_1$ et $X_2$ sur $\RP^1\setminus (0\cup \infty)$ par $t=(t')^{-1}$, $y=y'$, $z=z'$
et $u=(t')^{n/2}u'$, et on la note par $X$. Notons $T^*:=\Pic(X_{\bk})$. Alors on a $H^1(k,\Sym^2T^*)=0$.
\end{prop}

Notons $k':=k(\sqrt{a})$. Soient $\sigma\in \Gal(k'/k)$ l'\'el\'ement non-trivial, $(P_i(t))_{i\in I}$ les facteurs irr\'e\-ductibles de $P(t)$,
et $(e_{j_i})_{j_i\in J_i}$ les racines de $P_i(t)$ dans $\bk$. Soient $O_{i,i'}$ l'ensemble des $\Gamma_{k'}$-orbites de $J_i\times J_{i'}$ si $i\neq i'$, et $O_{i,i}$ l'ensemble des $\Gamma_{k'}$-orbites de $J_i\times J_i/\sim $, o\`u $\sim$ est donn\'e par $(j_i,j'_i)\sim (j'_i,j_i)$. Soit $\Delta_i$ l'orbite diagonale de $O_{i,i}$.
Pour une orbite $S$, on note $|S|$ le nombre des \'el\'ements de $S$. L'action de $\Gamma_k$ sur $J_i$ et $J_{i'}$ induit une action de $\sigma$ sur $O_{i,i'}$. Ainsi $\sigma (\Delta_i)=\Delta_i$.

\begin{lem}\label{orbitodd}
Si $|J_{i_0}|$ est impair, alors il existe au moins une orbite $S\in O_{i_0,i}$ pour chaque $i\neq i_0$, telle que $\sigma(S)=S$.
\end{lem}

\begin{proof}
Puisque l'action de $\Gamma_{k'}$ sur $J_i$ est transitive, pour chaque orbite $S\in O_{i_0,i}$, $|S|$ est divisible par $|J_i|$.
Donc si $S\neq \sigma(S)$, $|S\cup \sigma(S)|$ est divisible par $2\cdot |J_i|$.
Puisque $|J_{i_0}\times J_i|=|J_{i_0}|\cdot |J_i|$ et $|J_{i_0}|$ est impair, il existe au moins une orbite $S\in O_{i_0,i}$ telle que $\sigma(S)=S$.
\qed
\end{proof}

\begin{lem}\label{orbiteven}
Pour chaque $i$ et chaque orbite $S\in O_{i,i}$, on a $2\cdot |S|$ est divisible par $|J_i|$.
De plus, si $|J_i|$ est pair, alors il existe au moins une orbite $S\in O_{i,i}$, telle que $\sigma(S)=S$ et $|S|=s\cdot |J_i|/2$, o\`u $s$ est un entier impair.
\end{lem}

\begin{proof}
Soit $(j_i,j'_i)$ un \'el\'ement de l'orbite $S\in O_{i,i}$. Pour l'action de $\Gamma_{k'}$, soit $St_j$ (resp. $St_{j,j'}$)
le stabilisateur de $j_i$ (resp. de $(j_i,j'_{i})$). Ainsi s'il y a $g\in \Gamma_{k'}$ tel que $g(j_i)=j'_i$ et $g(j'_i)=j_i$, alors
$$St_{j,j'}=(St_j\cap (g\cdot St_j\cdot g^{-1}))\cup ((g\cdot St_j)\cap (g\cdot (g\cdot St_j\cdot g^{-1})))
=(St_j\cap (g\cdot St_j\cdot g^{-1}))\cup (g\cdot (St_j\cap(g\cdot St_j\cdot g^{-1})));$$
sinon, $St_{j,j'}$ est un sous-groupe de $St_{j}$. Ainsi $2\cdot [\Gamma_{k'}:St_{j,j'}]$ est divisible par $[\Gamma_{k'}:St_j]$, et donc $2\cdot |S|$ est divisible par $|J_i|$.

Supposons que $|J_i|$ est pair. On a $|J_i\times J_i/\sim|=|J_i||J_i-1|/2$. 
D'apr\`es l'argument du Lemme \ref{orbitodd}, il existe au moins une orbite $S\in O_{i,i}$, telle que $\sigma(S)=S$ et $|S|=s\cdot |J_i|/2$, o\`u $s$ est un entier impair.
\qed
\end{proof}

\begin{proof}[D\'emonstration de la Proposition \ref{conicsym}]
Soient $F$ le diviseur de $X_{\bk}$ d\'efini par $t=\infty$, $G$ le diviseur de $X_{\bk}$ d\'efini
par $(u,y-\sqrt{a}z)$, $G'$ le diviseur de $X_{\bk}$ d\'efini par $(u,y+\sqrt{a}z)$, $D_{j_i}$ le diviseur de $X_{\bk}$ d\'efini par
$(t-e_{j_i}, y-\sqrt{a}z)$ et $D'_{j_i}$ le diviseur de $X_{\bk}$ d\'efini par $(t-e_{j_i}, y+\sqrt{a}z)$, pour chaque $i\in I$ et chaque $j_i\in J_i$.
Pour chaque $i\in I$, on note $D_i:=\sum_{j_i\in J_i}D_{j_i}$
et $D'_i:=\sum_{j_i\in J_i}D'_{j_i}$.
Pour chaque orbite $S\in O_{i,i'}$,
on note 
$$D_S:=\sum_{(j_i,j_{i'})\in S}D_{j_i}\symprod D_{j_{i'}}\in \Sym^2\Pic(X_{\bk})\ \ \  \text{et}
\ \ \  D'_S:=\sum_{(j_i,j_{i'})\in S}D'_{j_i}\symprod D'_{j_{i'}}\in \Sym^2\Pic(X_{\bk}).$$

 D'apr\`es la d\'emonstration de \cite[Prop. 5.1]{CTCS} (cf. \cite[Section 7]{CTSSD} pour plus de d\'etails dans le cas o\`u $\deg(P)=4$), $T^*=\Pic(X_{\bk})$ est engendr\'e par $(D_{j_i})_{j_i\in \cup_iJ_i}$, $(D'_{j_i})_{j_i\in \cup_iJ_i}$, $F$, $G$ et $G'$,
 et les relations sont: $G'-G=\Sigma_{i}\Sigma_{j_i}D_{j_i}-\frac{\sum_i|J_i|}{2}F$ et $D_{j_i}+D'_{j_i}=F$ pour chaque $j_i$. 
 Ainsi $\Pic(X_{\bk})$ est un $\BZ$-module libre de base $(D_{j_i})_{j_i\in \cup_iJ_i}$, $F$ et $G$, et l'action de $\Gamma_{k'}$ sur cette base est de permutation. Donc l'action de $\Gamma_{k'}$ sur $\Sym^2T^*$ est aussi de  permutation. Ainsi $H^1(k',\Sym^2T^*)=0$. 
 Par la suite spectrale de Hochschild-Serre, on a:
 $$ H^1(\Gal(k'/k), (\Sym^2T^*)^{\Gamma_{k'}})\cong H^1(k,\Sym^2T^*).$$
 Ainsi il suffit de montrer que $H^1(\Gal(k'/k), (\Sym^2T^*)^{\Gamma_{k'}})=0$.

\bigskip

L'id\'ee de la d\'emonstration de l'\'egalit\'e $H^1(\Gal(k'/k), (\Sym^2T^*)^{\Gamma_{k'}})=0$ 
est de trouver une filtration $\Gal(k'/k)$-invariante de $(\Sym^2T^*)^{\Gamma_{k'}}$ telle que la cohomologie des quotients successifs est $0$.
En fait, on trouve des sous-groupes $A_l\sbt (\Sym^2T^*)^{\Gamma_{k'}}$, $l=1,\dots,m$,
 tels que $(\Sym^2T^*)^{\Gamma_{k'}}\iso \bigoplus_lA_l$, le sous-groupe $\bigoplus_{l=1}^{l_0}A_l\sbt (\Sym^2T^*)^{\Gamma_{k'}}$ soit $\Gal(k'/k)$-invariant,
et $H^1(k,A'_{l_0})=0$ pour tous les $l_0=1,\dots ,m$, o\`u $A'_{l_0}:=\bigoplus_{l=1}^{l_0}A_l/\bigoplus_{l=1}^{l_0-1}A_l$.

Puisque $\Pic(X_{\bk})$ est un $\BZ$-module libre de base $(D_{j_i})_{j_i\in \cup_iJ_i}$, $F$ et $G$, 
le $\BZ$-module $(\Sym^2T^*)^{\Gamma_{k'}}$ est libre de base
$F\symprod F$, $F\symprod G$, $G\symprod G$, $(F\symprod D_i)_{i\in I}$, $(G\symprod D_i)_{i\in I}$ et
$(D_S)_{S\in \cup_{i,i'\in I} O_{i,i'}}$.

D'apr\`es le Lemme \ref{orbiteven}, pour chaque $i$ avec $|J_i|$ pair, on choisit une orbite $S^0_{i}\in O_{i,i}$ telle que $\sigma(S^0_i)=S^0_i$ et
$|S^0_i|=s_i\cdot |J_i|/2$, o\`u $s_i$ est un entier impair. S'il existe au moins un $i_0$ tel que $|J_{i_0}|$ est impair,
on choisit une orbite $S^1_{i}\in O_{i_0,i}$ telle que $\sigma(S^1_i)=S^1_i$ pour chaque $i\neq i_0$ avec $|J_i|$ impair.

Dans la base de $(\Sym^2T^*)^{\Gamma_{k'}}$, on peut remplacer $D_{S^1_i}$ par $D_{i_0}\symprod D_i$ (si $i_0$ existe), pour tous les $i\neq i_0$ avec $|J_i|$ impair. On d\'efinit $(A_l)_{l=1}^6$:

$A_1$ est le $\BZ$-module libre de base $F\symprod F$, $(F\symprod D_i)_{i\in I,\ |J_i|\ impair}$ et $(D_{i_0}\symprod D_i)_{i\neq i_0,\ |J_i|\ impair}$ (si $i_0$ existe).

$A_2$ est le $\BZ$-module libre de base $(F\symprod D_i)_{i\in I,\ |J_i|\ pair}$ et $(D_{S^0_i})_{i\in I,\ |J_i|\ pair}$.

$A_3$ est le $\BZ$-module libre de base $G\symprod F$.

$A_4$ est le $\BZ$-module libre de base $(D_S)_{S\in O_0}$, o\`u
$$O_0:= \bigcup_{i,i'\in I}O_{i,i'}\setminus \{(\Delta_i)_i, (S^0_i)_{|J_i|\ pair}, (S^1_i)_{i\neq i_0,\ |J_i|\ impair}\}.$$

$A_5$ est le $\BZ$-module libre de base $(G\symprod D_i)_{i\in I}$ et $(D_{\Delta_i})_{i\in I}$.

$A_6$ est le $\BZ$-module libre de base $G\symprod G$.

On a $(\Sym^2T^*)^{\Gamma_{k'}}\iso \bigoplus_{l=1}^6A_l$. Puisque $G'-G=\Sigma_{i}\Sigma_{j_i}D_{j_i}-\frac{\sum_i|J_i|}{2}F$,
$D_{j_i}+D'_{j_i}=F$, $\sigma(S^0_i)=S^0_i$, $\sigma(S^1_i)=S^1_i$, $\sigma(\Delta_i)=\Delta_i$ et $O_0$ est $\sigma$-invariant,
on a que $\bigoplus_{l=1}^{l_0}A_l\sbt (\Sym^2T^*)^{\Gamma_{k'}}$ est $\Gal(k'/k)$-invariant pour chaque $l_0=1,\dots, 6$.

Pour tous les $1\leq l\leq 6$, notons $A'_{l_0}:=\bigoplus_{l=1}^{l_0}A_l/\bigoplus_{l=1}^{l_0-1}A_l.$
Alors $A'_{l_0}$ est un $\Gal(k'/k)$-module et
\begin{equation}\label{eq5.1}
A_{l_0}\hookrightarrow \bigoplus_{l=1}^{l_0}A_l \to A'_{l_0}
\end{equation}
 est un isomorphisme de groupes ab\'eliens.
 
\medskip

Il suffit de montrer que $H^1(\Gal(k'/k),A'_{l_0})=0$ pour tous les $1\leq l\leq 6$.

Pour $l_0=1$, si $i_0$ n'existe pas, on a $A_1=\BZ \cdot (F\symprod F)$ et $H^1(\Gal(k'/k),A_1)=0$.

Supposons que $i_0$ existe. Soit $B_1$ le groupe ab\'elien
engendr\'e librement par $F\symprod D_{i_0}$ et $F\symprod D'_{i_0}$. Ainsi $H^1(\Gal(k'/k),B_1)=0$. Soit $B'_1$ le sous-groupe de $A_1$ engendr\'e librement
par $|J_{i_0}|F\symprod F$, $F\symprod D_{i_0}$, $|J_{i_0}|F\symprod D_i$ pour tous les $i\neq i_0$ avec $|J_i|$ impair et $D_{i_0}\symprod D_i$ pour tous les
$i\neq i_0$ avec $|J_i|$ impair. Ainsi $|A_1/B'_1|=|J_{i_0}|^{M}$, o\`u $M$ est le nombre de $i\in I$ avec $J_i$ impair. Donc $|A_1/B'_1|$ est impair.
Puisque $F\symprod D'_{i_0}=|J_{i_0}|F\symprod F-F\symprod D_{i_0}$, $B_1$ est un sous-groupe de $B'_1$.

Notons $C_1:=B'_1/B_1$. Ainsi $C_1$ est engendr\'e librement par les classes de $|J_{i_0}|F\symprod D_i$ et de $D_{i_0}\symprod D_i$ pour tous les $i\neq i_0$ avec $|J_i|$ impair. On a
$$\sigma(D_{i_0}\symprod D_i)=D'_{i_0}\symprod D'_i=|J_{i}|F\symprod D'_{i_0}-|J_{i_0}|F\symprod D_i+D_{i_0}\symprod D_i.$$
Ainsi $C_1$ est engendr\'e librement par $D_{i_0}\symprod D_i$ et $\sigma(D_{i_0}\symprod D_i)$ pour tous les $i\neq i_0$ avec $|J_i|$ impair.
Donc $C_1$ est $\Gal(k'/k)$ invariant, et $H^1(\Gal(k'/k),C_1)=0$.

Puisque $H^1(\Gal(k'/k),B_1)=0$, on a $B'_1$ est $\Gal(k'/k)$ invariant, et $H^1(\Gal(k'/k),B'_1)=0$.
Donc $A_1/B'_1$ est $\Gal(k'/k)$ invariant, et $H^1(\Gal(k'/k),A_1/B'_1)=0$, parce que $|A_1/B'_1|$ est impair.
Donc $A_1$ est $\Gal(k'/k)$ invariant, et $H^1(\Gal(k'/k),A_1)=0$.

Pour \ $l_0=2$, \ soit \ $B_2$ \ un \ $\BZ$-module \ libre \ de \ base \ $(\sigma(D_{S^0_i}))_{i\in I,\ |J_i|\ pair}$ \ et \ $(D_{S^0_i})_{i\in I,\ |J_i|\ pair}$. \ On \ a \ $H^1(\Gal(k'/k), B_2)=0$.  On a
$$\sigma(D_{S^0_i})-D_{S^0_i}=\sum_{(j_i,j_{i'})\in S_i^0}((F-D_{j_i})\symprod (F-D_{j_{i'}})-D_{j_i}\symprod D_{j_{i'}})
=|S_i^0|^2F\symprod F-\sum_{(j_i,j_{i'})\in S_i^0}(D_{j_i}+ D_{j_{i'}})\symprod  F.$$
Puisque $\sigma(D_{S^0_i})-D_{S^0_i}$ est $\Gamma_{k'}$-invariant, $\sum_{(j_i,j_{i'})\in S_i^0}(D_{j_i}+ D_{j_{i'}})$ est $\Gamma_{k'}$-invariant et donc
$$\sigma(D_{S^0_i})-D_{S^0_i}=|S_i^0|^2F\symprod F-s_i F\symprod D_i.$$
Alors \ $\Im(B_2\sbt (A_1\oplus A_2)\to A'_2)$ \ est \ le \ sous-groupe \ engendr\'e \ par \ l'\'el\'ement \ $(s_i F\symprod D_i)_{i\in I,\ |J_i|\ pair}$ \ et $(D_{S^0_i})_{i\in I,\ |J_i|\ pair}$.
D'apr\`es \eqref{eq5.1}, la composition $B_2\sbt (A_1\oplus A_2)\to A'_2$ est injective, et $|A'_2/B_2|=\prod_{i\in I,\ |J_i|\ pair}s_i$. Ainsi $|A'_2/B_2|$ est impair,
et donc $H^1(\Gal(k'/k),A'_2/B_2)=0$. Ainsi $H^1(\Gal(k'/k),A'_2)=0$.

Pour $l_0=3$, on a
$$\sigma(F\symprod G)=F\symprod G'=F\symprod G+ \sum_iF\symprod D_i-\frac{\sum_i|J_i|}{2}F\symprod F.$$
Donc $(\sigma(F\symprod G)-F\symprod G)\in A_1\oplus A_2$. Alors l'action de $\Gal(k'/k)$ sur $A'_3$ est triviale, et donc $H^1(\Gal(k'/k),A'_3)=0$.

Soit $(F\symprod E)\in (\Sym^2T^*)^{\Gamma_{k'}}$, o\`u $E\in T^*$. Alors $E\in (T^*)^{\Gamma_{k'}}$. Ainsi $E$ est une combinaison de $G$, $F$ et $(D_i)_{i\in I}$, alors $(F\symprod E)\in A_1\oplus A_2\oplus A_3$.

Pour $l_0=4$, on a $(D'_{j_i}\symprod D'_{j_{i'}})=D_{j_i}\symprod D_{j_{i'}}+F\symprod (F-D_{j_i}-D_{j_{i'}})$ et donc pour chaque orbite $S$,
$(\sigma (D_S)-D_{\sigma(S)})\in A_1\oplus A_2\oplus A_3$. Ainsi dans $(\oplus_{l=1}^4A_l)/(\oplus_{l=1}^3A_l)$, on a $\sigma (D_S)=D_{\sigma(S)}$.
Donc l'action de $\Gal(k'/k)$ sur $A'_4$ est une permutation, et $H^1(\Gal(k'/k),A'_4)=0$.

Pour $l_0=5$, en utilisant $G'-G=\Sigma_{i}\Sigma_{j_i}D_{j_i}-\frac{\sum_i|J_i|}{2}F$, on trouve
$$(G'\symprod D_i-D_{\Delta_i}-G\symprod D_i)\in A_1\oplus A_2\oplus A_4.$$
Ainsi $A'_5$ est engendr\'e librement par $G'\symprod D_i$ et $G\symprod D_i$ pour tous les $i\in I$.
On a 
$$\sigma(G'\symprod D_i)=G\symprod D'_i=-G\symprod D_i+ |J_i|G\symprod F.$$ 
Donc dans $A'_5$, on a $\sigma(G'\symprod D_i)=-G\symprod D_i$. Donc $H^1(\Gal(k'/k),A'_5)=0$.

Pour $l_0=6$, en utilisant $G'-G=\Sigma_{i}\Sigma_{j_i}D_{j_i}-\frac{\sum_i|J_i|}{2}F$, on a
$$(\sigma(G\symprod G)-G\symprod G)\in \bigoplus_{l=1}^5A_l.$$
Donc dans $A'_6$, $\sigma(G\symprod G)=G\symprod G$, et $H^1(\Gal(k'/k),A'_6)=0$.
\qed
\end{proof}

\section{Surfaces de del Pezzo} 

Dans cette section, soient  $X$ une surface de del Pezzo sur $k$ et  $\CT$ un torseur universel de $X$. 
On cherche \`a conr\^oler  l'exposant de torsion du quotient $\frac{H^3_{\nr}(\CT,\BQ/\BZ(2))}{H^3(k,\BQ/\BZ(2))}$.

\medskip

Une surface projective, lisse, g\'eom\'etriquement connexe $X$ est appel\'ee
 \emph{surface de del Pezzo} si le faisceau anticanonique
 $-K_X$ est ample. Le degr\'e d'une telle surface $X$ est  $deg(X):=(K_X,K_X)$. 
Par \cite[Exercise 3.9]{K2}, $X$ est alors g\'eom\'etriquement rationnelle, on a $1\leq deg(X)\leq 9$ et $Pic(X_{\bk})\iso \BZ^{10-deg(X)}$.

Soit $X$ une surface de del Pezzo de degr\'e $d\leq 6$. 
Soit $r=9-d$.
Rappelons des r\'esultats dans \cite[Chapter 4]{Ma}.
On sait  (\cite[Prop. 25.1]{Ma}) 
qu'il existe une base $(l_i)_{i=0}^r$ de $\Pic(X_{\bk})$ telle que 
\begin{equation}\label{basedePic}
(l_0,l_0)=1,\ \ \  (l_i,l_i)=-1\  \text{si}\  i\neq 0,\ \ (l_i,l_j)=0\ \text{si}\ i\neq j\ \ \ \text{et}\ \ \   
\omega=-3l_0+\sum_{i=1}^rl_i,
\end{equation}
o\`u $\omega:=[K_X]\in \Pic(X_{\bk})$.
Soit  $\mathfrak{S}_r$ le groupe sym\'etrique d'indice $r$. 
L'action de $\mathfrak{S}_r$ sur $\{l_i\}_{i=0}^r $, qui  pr\'eserve $l_0$ et permute $(l_i)_{i=1}^r$, induit une action de permutation de $\mathfrak{S}_r$ sur $\Pic(X_{\bk})$.
Pour $r=3,4,5,6,7,8$, soient $R_r$ le syst\`eme de racines de type $A_1\times A_2,A_4,D_5,E_6,E_7,E_8$ respectivement.
Par  \cite[Thm. 23.9]{Ma}, on a:
\begin{itemize}
\item[\rm (1)] le groupe de Weyl $W(R_r)$ agit sur $\Pic(X_{\bk})$ et cette action pr\'eserve $\omega$ et l'intersection;
\item[\rm (2)] l'action de $\Gamma_k$ sur $\Pic(X_{\bk})$, qui pr\'eserve $\omega$ et l'intersection, se factorise par  le groupe de Weyl $W(R_r)$;
\item[\rm (3)] l'action de $\mathfrak{S}_r$ sur $\Pic(X_{\bk})$ ci-dessus, qui pr\'eserve $\omega$ et l'intersection aussi, induit une inclusion canonique $\mathfrak{S}_r\sbt W(R_r)$ ne d\'ependant que du choix de la base $(l_i)_{i=0}^r$.
\end{itemize}

\begin{prop}\label{delpezzo5}
Soient \ $X$ \ une \ surface \ de \ del \ Pezzo \ de \ degr\'e \ $\geq 5$ \ et \ $\CT$ \ un \ torseur \ universel \ de \ $X$. \ Alors $\frac{H^3_{\nr}(\CT,\BQ/\BZ(2))}{H^3(k,\BQ/\BZ(2))}=0$.
\end{prop}

\begin{proof}
Par un r\'esultat classique (cf. \cite[Lem. 3.2]{C2}), le $\Gamma_k$-module $\Pic(X_{\bar{k}})$ est stablemement de permutation. D'apr\`es (\ref{symoplus}), on a $H^1(k,\Sym^2\Pic(X_{\bk}))=0$.
Par le th\'eor\`eme \ref{H4Z2} et sa preuve, il suffit de montrer que l'on a $X(k)\neq \emptyset$  ou  que le morphisme
$(\Sym^2T^*)^{\Gamma_k}\xrightarrow{\cup}\CCH^2(X_{\bk})\cong \BZ$
 dans le th\'eor\`eme \ref{H4Z2} (2)  est surjectif. 

Notons $d$ le degr\'e de $X$.

Si $d=5$ ou $7$, par les travaux de Enriques, Ch\^atelet,  Manin, Swinnerton-Dyer (voir \cite[Section 4]{CT99} ou \cite[Th\'eor\`eme 2.1]{VA}), on a $X(k)\neq\emptyset$.

Si $d=9$, on a $\Pic(X_{\bk})^{\Gamma_k}=\Pic(X_{\bk})$. 
Puisque $X_{\bk}\cong \BP^2_{\bk}$, l'homomorphisme $\cup$ est surjectif.

Si $d=8$, on a  l'une des possibilit\'es suivantes (voir par exemple \cite[Exemples 3.1.3 et 3.1.4]{AB}):
\begin{itemize}
\item[\rm (i)] $X$ est un \'eclatement de $\BP^2_{k}$ en un $k$-point, et dans ce cas, $X(k)\neq \emptyset$.
\item[\rm (ii)] Il existe des coniques lisses $C_1$, $C_2$ sur $k$ telles que $X\iso C_1\times C_2$. Dans ce cas, on a $\Pic(X_{\bk})^{\Gamma_k}=\Pic(X_{\bk})$, $X_{\bk}\cong \BP^1_{\bk}\times \BP^1_{\bk}$ et donc l'homomorphisme $\cup$ est surjectif.
\item[\rm (iii)] Il existe une extension de corps $K/k$ de degr\'e $2$ et une conique $C$ sur $K$ telles que $X\iso R_{K/k}C$, o\`u $R_{K/k}$ d\'esigne la restriction \`a  la Weil de $K$ \`a $k$.
Dans ce cas, il existe $D_1,D_2\in \Pic(X_{\bk})$ tels que $\Gamma_k$ permute $D_1,D_2$ et l'intersection de $D_1$ et $D_2$ est un point. 
Alors $D_1\symprod D_2\in (\Sym^2\Pic(X_{\bk}))^{\Gamma_k}$, $\cup (D_1\symprod D_2)=1$ et donc l'homomorphisme $\cup$ est surjectif.
\end{itemize}

Si $d=6$, avec les notations ci-dessus, soit (cf. \cite[\S 25.5.7]{Ma}) 
$$\sigma: \Pic(X_{\bk})\to \Pic(X_{\bk}):\sum_{i=0}^3a_il_i\mapsto (a_0+c)l_0+\sum_{i=1}^3 (a_i-c)l_i\ \ \ \text{avec}\ \ \ c:=a_0+a_1+a_2+a_3,$$
et on a $W(R_3)\cong \mathfrak{S}_3\times \BZ/2\cdot \sigma$.
Soit $L_1:=\sum_{1\leq i<j\leq 3}(l_0-l_i)\symprod (l_0-l_j)$ et $L_2:=l_0\symprod (\omega+l_0)$. 
Puisque l'action de $\Gamma_k$ se factorise par $W(R_3)$,
on peut calculer que $L_1,L_2 \in (\Sym^2\Pic(X_{\bk}))^{\Gamma_k}$ et $\cup (L_1)=3$, $\cup (L_2)=-2$.
Alors $\cup$ est surjectif.
\qed
\end{proof}

\begin{lem}\label{delpezzo3lem}
Soient  $X$ une surface de del Pezzo de degr\'e $d=1,2,3,4$ et  $\CT$ un torseur universel de $X$. 
Alors
$$\frac{H^3_{\nr}(\CT,\BQ/\BZ(2))}{H^3(k,\BQ/\BZ(2))}[p]=0\ \ \ \text{pour tout nombre premier}\ \ \ 
\begin{cases} p\neq 2\ \ &\text{si} \ \ d=4   \\ p\neq 2,3\ \ &\text{si}\ \ d=2,3 \\ p\neq 2,3,5\ \ &\text{si} \ \ d=1  \end{cases}$$
\end{lem}

\begin{proof}
Soit $r=9-d$.
Avec les notations ci-dessus, 
d'apr\`es \cite[\S 26.6]{Ma}, $\#W(R_r)=2^73\cdot 5$, $2^73^45,$ $2^{10}3^45\cdot 7$, $2^{14}3^55^27 $ pour $d=4,3,2,1$ respectivement.
Alors, pour tout nombre premier $p$ de l'\'enonc\'e, tout $p$-groupe de Sylow de $\mathfrak{S}_r$ est un  $p$-groupe de Sylow de $W(R_r)$.
Dans ce cas, l'action de tout $p$-sous-groupe de $W(R_r)$ sur $\Pic(X_{\bk})$ est donc de permutation.
C'est donc aussi le cas sur $\Sym^2\Pic(X_{\bk})$.
Par un  argument de corestriction-restriction, $H^1(k,\Sym^2\Pic(X_{\bk}))[p]=0$.

On consid\`ere le morphisme dans le Th\'eor\`eme \ref{H4Z2} (2): $(\Sym^2T^*)^{\Gamma_k}\xrightarrow{\cup}\CCH^2(X_{\bk})\cong \BZ$. 
Par la d\'efinition, $\cup (\omega\symprod \omega)=4,3,2,1$ pour $d=4,3,2,1$ respectivement.
On conclut alors avec le Th\'eor\`eme \ref{H4Z2} (2). 
\qed
\end{proof}

\begin{rem}{\rm
Sous les hypoth\`eses du Lemme \ref{delpezzo3lem}, par la m\^eme m\'ethode, on peut montrer le r\'esultat ci-dessous (qui est d\'ej\`a connu):
$$\frac{\Br(X)}{\Br(k)}[p]=0\ \ \ \text{pour tout nombre premier}\ \ \ 
\begin{cases} p\neq 2\ \ &\text{si} \ \ d=4   \\ p\neq 2,3\ \ &\text{si}\ \ d=2,3 \\ p\neq 2,3,5\ \ &\text{si} \ \ d=1  \end{cases}$$}
\end{rem}

\begin{thm}\label{delpezzo3}
Soient $X$ une surface de del Pezzo de degr\'e $3$ ou $4$ et $\CT$ un torseur universel de $X$. 
 Alors $\frac{H^3_{\nr}(\CT,\BQ/\BZ(2))}{H^3(k,\BQ/\BZ(2))}[p]=0$ pour tout $p\neq 2$.
\end{thm}

\begin{proof}
D'apr\`es le Lemme \ref{delpezzo3lem}, il suffit de montrer que $\frac{H^3_{\nr}(\CT,\BQ/\BZ(2))}{H^3(k,\BQ/\BZ(2))}[3]=0$
pour toute  surface de del Pezzo $X$ de degr\'e $3$.

On suppose donc que le degr\'e de $X$ est $3$.

Notons $T^*:=\Pic(X_{\bk})$.
Soient $(l_i)_{i=0}^6$ une base dans (\ref{basedePic}) et $\omega$ le fibr\'e en droites canonique de $X$.
Il y a 27 courbes exceptionnelles  (les 27 droites) sur $X_{\bk}$. 
Le groupe $\Gamma_k$ agit par permutation de ces droites. On les note   $(D_i)_{i=1}^{27}$.
D'apr\`es \cite[prop. 26.1]{Ma}, on peut supposer   $D_i=l_i$ pour $1\leq i\leq 6$, $D_i=2l_0-\sum_{j\neq i-6}l_j$ pour $7\leq i\leq 12$ et 
$(D_i)_{i=13}^{27}=(l_0-l_{i}-l_{j})_{1\leq i<j\leq 6}$.

Dans $\Sym^2T^*$, on note:
$$L:=(l_0\symprod l_0)-\sum_{i=1}^6(l_i\symprod l_i),$$ 
et on peut calculer que $6L=5(\omega\symprod \omega)-\sum_{i=1}^{27}(D_i\symprod D_i)$. Donc l'action de $\Gamma_k$ sur $L$ est triviale.

On consid\`ere le morphisme dans le Th\'eor\`eme \ref{H4Z2} (2): $(\Sym^2T^*)^{\Gamma_k}\xrightarrow{\cup}\CCH^2(X_{\bk})\cong \BZ$. 
Puisque $\cup (\omega\symprod \omega)=3 $ et  $\cup (L)= 7$, le morphisme $\cup$ est surjectif.
 Par le Th\'eor\`eme \ref{H4Z2} (ii), il suffit de montrer que $H^1(k,\Sym^2T^*)[3]=0$. 
 
 \medskip

Il existe une matrice $A\in M_{28\times 28}(\BZ )$, telle que dans le r\'eseau $\Sym^2T^*$, de rang $28$,  on ait l'\'egalit\'e de vecteurs  :
\begin{equation}\label{matrix}
[-L, (D_i\symprod D_i)_{i=1}^{27}]^t=A \cdot [(l_i\symprod l_i)_{i=0}^6, (-l_0\symprod l_i)_{i=1}^6,(l_i\symprod l_j)_{1\leq i<j\leq 6}]^t. 
\end{equation}

%\newpage

 Un calcul direct donne :
    $A=$
$$\xymatrix@R=0.2pc@C=0.3pc{
[-1&1&1&1&1&1&1&0&0&0&0&0&0&0&0&0&0&0&0&0&0&0&0&0&0&0&0&0\\
0&1&0&0&0&0&0&0&0&0&0&0&0&0&0&0&0&0&0&0&0&0&0&0&0&0&0&0\\
0&0&1&0&0&0&0&0&0&0&0&0&0&0&0&0&0&0&0&0&0&0&0&0&0&0&0&0\\
0&0&0&1&0&0&0&0&0&0&0&0&0&0&0&0&0&0&0&0&0&0&0&0&0&0&0&0\\
0&0&0&0&1&0&0&0&0&0&0&0&0&0&0&0&0&0&0&0&0&0&0&0&0&0&0&0\\
0&0&0&0&0&1&0&0&0&0&0&0&0&0&0&0&0&0&0&0&0&0&0&0&0&0&0&0\\
0&0&0&0&0&0&1&0&0&0&0&0&0&0&0&0&0&0&0&0&0&0&0&0&0&0&0&0\\
4&0&1&1&1&1&1&0&4&4&4&4&4&0&0&0&0&0&2&2&2&2&2&2&2&2&2&2\\
4&1&0&1&1&1&1&4&0&4&4&4&4&0&2&2&2&2&0&0&0&0&2&2&2&2&2&2\\
4&1&1&0&1&1&1&4&4&0&4&4&4&2&0&2&2&2&0&2&2&2&0&0&0&2&2&2\\
4&1&1&1&0&1&1&4&4&4&0&4&4&2&2&0&2&2&2&0&2&2&0&2&2&0&0&2\\
4&1&1&1&1&0&1&4&4&4&4&0&4&2&2&2&0&2&2&2&0&2&2&0&2&0&2&0\\
4&1&1&1&1&1&0&4&4&4&4&4&0&2&2&2&2&0&2&2&2&0&2&2&0&2&0&0\\
1&1&1&0&0&0&0&2&2&0&0&0&0&2&0&0&0&0&0&0&0&0&0&0&0&0&0&0\\
1&1&0&1&0&0&0&2&0&2&0&0&0&0&2&0&0&0&0&0&0&0&0&0&0&0&0&0\\
1&1&0&0&1&0&0&2&0&0&2&0&0&0&0&2&0&0&0&0&0&0&0&0&0&0&0&0\\
1&1&0&0&0&1&0&2&0&0&0&2&0&0&0&0&2&0&0&0&0&0&0&0&0&0&0&0\\
1&1&0&0&0&0&1&2&0&0&0&0&2&0&0&0&0&2&0&0&0&0&0&0&0&0&0&0\\
1&0&1&1&0&0&0&0&2&2&0&0&0&0&0&0&0&0&2&0&0&0&0&0&0&0&0&0\\
1&0&1&0&1&0&0&0&2&0&2&0&0&0&0&0&0&0&0&2&0&0&0&0&0&0&0&0\\
1&0&1&0&0&1&0&0&2&0&0&2&0&0&0&0&0&0&0&0&2&0&0&0&0&0&0&0\\
1&0&1&0&0&0&1&0&2&0&0&0&2&0&0&0&0&0&0&0&0&2&0&0&0&0&0&0\\
1&0&0&1&1&0&0&0&0&2&2&0&0&0&0&0&0&0&0&0&0&0&2&0&0&0&0&0\\
1&0&0&1&0&1&0&0&0&2&0&2&0&0&0&0&0&0&0&0&0&0&0&2&0&0&0&0\\
1&0&0&1&0&0&1&0&0&2&0&0&2&0&0&0&0&0&0&0&0&0&0&0&2&0&0&0\\
1&0&0&0&1&1&0&0&0&0&2&2&0&0&0&0&0&0&0&0&0&0&0&0&0&2&0&0\\
1&0&0&0&1&0&1&0&0&0&2&0&2&0&0&0&0&0&0&0&0&0&0&0&0&0&2&0\\
1&0&0&0&0&1&1&0&0&0&0&2&2&0&0&0&0&0&0&0&0&0&0&0&0&0&0&2].
}$$
Par calcul direct (et aussi par ordinateur) on \'etablit $\det(A)=5\cdot 2^{27}$. 
D'apr\`es (\ref{matrix}), le $\BZ_3$-module $(\Sym^2T^*)\otimes \BZ_3$ est libre avec une base: $-L, (D_i\symprod D_i)_{i=1}^{27}$.
Puisque le groupe $\Gamma_k$ agit par permutation des vecteurs $-L, (D_i\symprod D_i)_{i=1}^{27}$, on a
$$H^1(k,(\Sym^2T^*)\otimes \BZ_3)\cong 0\ \ \ \text{et}\ \ \ H^1(k,\Sym^2T^*)[3]=0.$$
\qed
\end{proof}

\begin{cor}\label{delP23}
Soient $X$ une surface de del Pezzo de degr\'e $2$.
 $\CT$ un torseur universel de $X$ et $\CT^c$ une compactification lisse de $\CT$. 
 Alors $\frac{H^3_{\nr}(\CT^c,\BQ/\BZ(2))}{H^3(k,\BQ/\BZ(2))}[p]=0$ pour tout $p\neq 2$.
\end{cor}

\begin{proof}
D'apr\`es le Lemme \ref{delpezzo3lem},
il suffit de montrer $\frac{H^3_{\nr}(\CT^c,\BQ/\BZ(2))}{H^3(k,\BQ/\BZ(2))}[3]=0$.

Sous les notations ci-dessus, on a l'inclusion canonique $W(R_6)\sbt W(R_7)$ et les sous-$p$-groupes de Sylow de $W(R_6)$ et de $W(R_7)$ sont isomorphes pour tout $p\neq 2$.

Dans le cas d'une surface de del Pezzo de degr\'e 2, l'action de $\Gamma_k$ sur $\Pic(X_{\bk})$ se factorise par $W(R_7)$, et donc, apr\`es conjugaison, il existe une extension de corps $k\sbt k'$ avec $([k':k],p)=1$ telle que l'action de $\Gamma_{k'}$ sur $\Pic(X_{\bk})$ se factorise par  $W(R_6)\sbt W(R_7)$. 
Par \cite[Cor. 26.7]{Ma}, dans $X_{k'}$, il existe une courbe exceptionnelle d\'efinie sur $k'$, donc il existe une surface de del Pezzo de degr\'e 3 $X'$ sur $k'$ et un $k'$-morphisme propre, surjectif, birationnel $X_{k'}\ra X'$. 
D'apr\`es le Th\'eor\`eme \ref{delpezzo3} et un argument de restriction-inflation,  on a $\frac{H^3_{\nr}(\CT^c,\BQ/\BZ(2))}{H^3(k,\BQ/\BZ(2))}[3]=0$.
\qed
\end{proof}

\begin{exam}{\rm
Soit $k$ un corps contenant une racine cubique de l'unit\'e non-triviale.
Soit $X$ une surface cubique diagonale d'\'equation:
$$ ax^3+by^3+cz^3+dt^3=0$$
en les variables $x,y,z,t$, avec $a,b,c,d\in k^{\times}$.
Soit $K:=k(\sqrt[3]{b/a},\sqrt[3]{ad/bc})$. 
Alors $X_K$ est  une $K$-surface $K$-rationnelle.
D'apr\`es le Corollaire \ref{zerorationnel}, $\frac{H^3_{\nr}(\CT^c,\BQ/\BZ(2))}{H^3(k,\BQ/\BZ(2))}$ est annul\'e par $9$.
D'apr\`es le Th\'eor\`eme \ref{delpezzo3}, on a $\frac{H^3_{\nr}(\CT^c,\BQ/\BZ(2))}{H^3(k,\BQ/\BZ(2))}=0$.}
\end{exam}

\begin{thm}\label{sauf2prime}
Soit $X$ une $k$-surface projective, lisse, g\'eom\'etriquement rationnelle.
Soit $\mathcal{T} \to X$ un torseur universel sur $X$ et soit $\mathcal{T}^c$
une $k$-compactification lisse de  $\mathcal{T}$.
Supposons que la surface $X$ est $k$-birationnellement \'equivalente \`a une surface de del Pezzo de degr\'e $\geq 2$
 ou \`a une surface fibr\'ee en coniques au-dessus d'une conique. 
Alors $\frac{H^3_{\nr}(\CT^c,\BQ/\BZ(2))}{H^3(k,\BQ/\BZ(2))}[p]=0$ pour tout $p\neq 2$.
\end{thm}

\begin{proof}
Soit $X$  une surface fibr\'ee en coniques au-dessus d'une conique $C$. 
Il existe une extension $K_1/k$ de degr\'e $2$ telle que $C_{K_1}\cong \BP^1_{K_1}$.
Notons $\eta$ le point g\'en\'erique de $C$ et $X_{\eta}$ sa fibre.
Alors $X_{\eta}$ est une conique, et donc elle d\'efinit $[X_{\eta}]\in \Br(k(C))[2]$.
Puisque $\Br(\bk(C))=0$, il existe une extension finie galoisienne $k'/k$ tel que $[X_{\eta}]|_{k'(C)}=0$.
Soient et $k\sbt K_2\sbt k'$ l'extension  correspondant \`a un $p$-sous-groupe de Sylow de $\Gal(k'/k)$.
Alors $p\nmid [K_2:k]$ et, par un argument de restriction-inflation, $[X_{\eta}]|_{K_2(C)}=0$.
Soit $K:=K_1\cdot K_2$. 
On a $p\nmid [K:k]$, $C_K\cong \BP^1_K$ et $X_{\eta}\times_kK\cong \BP^1_{K(C)}$.
Donc $X_K$ est  $K$-rationnelle.
D'apr\`es le Corollaire \ref{zerorationnel}, l'\'enonc\'e vaut pour une telle surface $X$.

Soit $X$ une surface de del Pezzo de degr\'e $\geq 2$. 
D'apr\`es  la Proposition \ref{delpezzo5}, le Th\'eor\`eme \ref{delpezzo3} et le Corollaire \ref{delP23}, l'\'enonc\'e vaut pour une telle surface $X$.

En g\'en\'eral, soit $X_1$ une surface projective, lisse, g\'eom\'etriquement rationnelle telle que $X$ et $X_1$ soient birationnellement \'equivalents. 
Il existe un torseur universel $\CT_1^c$ de $X_1$ tel que $\CT^c$ et $\CT_1^c$ soient stablement birationnellement \'equivalents par \cite[Prop. 2.9.2]{CTS}.
Donc  $H_{\nr}^3(\CT^c,\BQ/\BZ(2))\iso H_{\nr}^3(\CT_1^c,\BQ/\BZ(2))$ et on peut remplacer $X$ par $X_1$.
Donc on peut supposer que $X$ est une surface de del Pezzo de degr\'e $\geq 2$ ou une surface fibr\'ee en coniques au-dessus d'une conique.
On obtient le th\'eor\`eme en rassemblant les r\'esultats ci-dessus.
\qed
\end{proof}

Par la classification $k$-birationnelle des surfaces projectives, lisses (cf. \cite[Thm. 2.1]{K2}), 
toute surface projective, lisse, g\'eom\'etriquement rationnelle est
$k$-birationnellement \'equivalent \`a soit une surface de del Pezzo,
soit une surface fibr\'ee en coniques au-dessus d'une conique.
Dans le Th\'eor\`eme \ref{sauf2prime}, le seul cas que l'on ne traite pas est alors le cas o\`u $X$ est une surface de del Pezzo $k$-minimale de degr\'e 1.

\subsection*{Remerciements}
Nous remercions Jean-Louis Colliot-Th\'el\`ene pour plusieurs discussions.
Je remercie vivement les rapporteurs de l'\'Epijournal de G\'eom\'etrie Alg\'ebrique pour leurs commentaires.
Projet soutenu par l'attribution d'une allocation de recherche R\'egion Ile-de-France.

\providecommand{\bysame}{\leavevmode\hbox to3em{\hrulefill}\thinspace}
%\providecommand{\MR}{\relax\ifhmode\unskip\space\fi MR }
%\providecommand{\MRhref}[2]{%
%  \href{http://www.ams.org/mathscinet-getitem?mr=#1}{#2}
%}
%\providecommand{\href}[2]{#2}
%\begin{thebibliography}{{Kem}92}
%
%
%\end{thebibliography}

\bibliographystyle{amsalpha}
\bibliographymark{References}
% This would change the heading "References" by "Bibliography"
% \renewcommand{\refname}{Bibliography}
\def\cprime{$'$}

\end{document}